\documentclass[12pt,final]{amsart}
\usepackage{hyperref}

\usepackage{xspace,amssymb,epsfig,euscript,enumerate,amsmath,nicefrac}
\DeclareMathAlphabet{\mathbbm}{U}{bbm}{m}{n}

\usepackage[auth-sc-lg,affil-sl]{authblk}

\usepackage{tikz,etex,xspace,afterpage}
\usepackage{epstopdf,ifpdf, enumitem}
\usepackage{anindex}
\anindexsetup{
    heading,
    lines,
      title = {Index of Notation}
}
\usepackage[russian,english]{babel}

\usepackage{stackengine}

\stackMath
\usepackage[notref,notcite]{showkeys}
\usepackage[margin=3cm,footskip=25pt,headheight=20pt]{geometry}

\newcommand{\doubletilde}[1]{
  \tilde{{\tilde{#1}}}}
\usepackage{fancyhdr,mathrsfs}

\usepackage[nocfg]{nomencl}
\makenomenclature
  \usepackage{pdfsync}

\begin{document}
\newtheoremstyle{all}{11pt}{11pt}{\slshape}{}{\bfseries}{}{.5em}{}

\theoremstyle{all}

\newtheorem{itheorem}{Theorem}
\newtheorem{theorem}{Theorem}[section]
\newtheorem*{theoremfourfive}{Theorem 4.5}
\newtheorem*{proposition*}{Proposition}
\newtheorem{proposition}[theorem]{Proposition}
\newtheorem{corollary}[theorem]{Corollary}
\newtheorem{lemma}[theorem]{Lemma}
\newtheorem{assumption}[theorem]{Assumption}
\newtheorem{definition}[theorem]{Definition}
\newtheorem{ques}[theorem]{Question}
\newtheorem{conj}[theorem]{Conjecture}

\theoremstyle{remark}
\newtheorem{remark}[theorem]{Remark}
\newtheorem{example}{Example}[section]
\renewcommand{\theexample}{{\arabic{section}.\roman{example}}}
\newcommand{\nc}{\newcommand}
\newcommand{\renc}{\renewcommand}
\numberwithin{equation}{section}
\renc{\theequation}{\arabic{section}.\arabic{equation}}

\newcounter{subeqn}
\renewcommand{\thesubeqn}{\theequation\alph{subeqn}}
\newcommand{\subeqn}{\refstepcounter{subeqn}\tag{\thesubeqn}}\makeatletter
\@addtoreset{subeqn}{equation}
\newcommand{\newseq}{\refstepcounter{equation}}
  \nc{\kac}{\kappa^C}
  \newcommand{\ttalg}{{\it \doubletilde{\mathbb{T}}}\xspace}
\nc{\alg}{T}
\nc{\Lco}{L_{\la}}
\nc{\qD}{q^{\nicefrac 1D}}
\nc{\ocL}{M_{\la}}
\nc{\excise}[1]{}
\nc{\Dbe}{D^{\uparrow}}
\nc{\Dfg}{D^{\mathsf{fg}}}
\nc{\KLRwei}{\EuScript{W}}
\nc{\coset}{\EuScript{W}}
\nc{\zero}{o}
\nc{\defr}{\operatorname{def}}
\nc{\op}{\operatorname{op}}
\nc{\Sym}{\operatorname{Sym}}
\nc{\Symt}{S}
\nc{\hatD}{\widehat{\Delta}}
\nc{\tr}{\operatorname{tr}}
\newcommand{\Mirkovic}{Mirkovi\'c\xspace}
\nc{\tla}{\mathsf{t}_\la}
\nc{\llrr}{\langle\la,\rho\rangle}
\nc{\lllr}{\langle\la,\la\rangle}
\nc{\K}{\mathbbm{k}}
\nc{\Stosic}{Sto{\v{s}}i{\'c}\xspace}
\nc{\cd}{\mathcal{D}}
\nc{\cT}{\mathcal{T}}
\nc{\vd}{\mathbb{D}}
\nc{\lift}{\gamma}
\nc{\cox}{h}
\nc{\Aut}{\operatorname{Aut}}
\nc{\R}{\mathbb{R}}
\renc{\wr}{\operatorname{wr}}
\newcommand{\MTS}{\mathsf{MTS}}
  \nc{\Lab}[2]{\La^{#1}_{#2}}
  \nc{\Lamvwy}{\Lam\Bv\Bw\By}
  \nc{\Labwv}{\Lab\Bw\Bv}
  \nc{\nak}[3]{\mathcal{N}(#1,#2,#3)}
  \nc{\hw}{highest weight\xspace}
  \nc{\al}{\alpha}
\newcommand{\Gz}{G^0}
\newcommand{\Gc}[1]{G_{#1}}

\newcommand{\dgmod}{\operatorname{-dg-mod}}
\nc{\gmod}{\operatorname{-gmod}}
  \nc{\be}{\beta}
  \nc{\bM}{\mathbf{m}}
  \nc{\Bu}{\mathbf{u}}

  \nc{\bkh}{\backslash}
  \nc{\Bi}{\mathbf{i}}
  \nc{\Bm}{\mathbf{m}}
  \nc{\Bj}{\mathbf{j}}
 \nc{\Bk}{\mathbf{k}}
 \nc{\poly}{\mathscr{P}}
 \nc{\Flag}{\mathsf{Flag}}
 \nc{\Perv}{\mathsf{Perv}}
\newcommand{\bS}{\mathbb{S}}
\newcommand{\bT}{\mathbb{\tilde{T}}}
\newcommand{\bt}{\mathbbm{t}}
\nc{\dggmod}{\operatorname{-dggmod}}

\nc{\bd}{\mathbf{d}}
\nc{\second}{\tau}
\nc{\D}{\mathcal{D}}
\nc{\mmod}{\operatorname{-mod}}  
\newcommand{\red}{\mathfrak{r}}

\nc{\RAA}{R^\A_A}
  \nc{\Bv}{\mathbf{v}}
  \nc{\Bw}{\mathbf{w}}
\nc{\Id}{\operatorname{Id}}
\nc{\Cth}{S_h}
\nc{\Cft}{S_1}
\def\MHM{{\operatorname{MHM}}}

\newcommand{\cD}{\mathcal{D}}
\newcommand{\LCP}{\operatorname{LCP}}
  \nc{\By}{\mathbf{y}}
\nc{\eE}{\EuScript{E}}
  \nc{\Bz}{\mathbf{z}}
  \nc{\coker}{\mathrm{coker}\,}
  \renc{\C}{\mathbb{C}}
\nc{\ab}{{\operatorname{ab}}}
\nc{\wall}{\mathbbm{w}}
  \renc{\ch}{\mathrm{ch}}
  \nc{\de}{\delta}
  \nc{\ep}{\epsilon}
  \nc{\Rep}[2]{\mathsf{Rep}_{#1}^{#2}}
  \nc{\Ev}[2]{E_{#1}^{#2}}
  \nc{\fr}[1]{\mathfrak{#1}}
  \nc{\fp}{\fr p}
  \nc{\fq}{\fr q}
  \nc{\fl}{\fr l}
  \nc{\fgl}{\fr{gl}}
\nc{\rad}{\operatorname{rad}}
\nc{\Ind}{\operatorname{Ind}}
\newcommand{\sS}{\mathsf{S}} 
\newcommand{\sJ}{\mathsf{J}} 
\newcommand{\sT}{\mathsf{T}}
\newcommand{\cP}{\mathcal{P}}
\newcommand{\cellb}{a}
  \nc{\GL}{\mathrm{GL}}
\newcommand{\arxiv}[1]{\href{http://arxiv.org/abs/#1}{\tt arXiv:\nolinkurl{#1}}}
  \nc{\Hom}{\mathrm{Hom}}
  \nc{\im}{\mathrm{im}\,}
  \nc{\La}{\Lambda}
  \nc{\la}{\lambda}
  \nc{\mult}{b^{\mu}_{\la_0}\!}
  \nc{\mc}[1]{\mathcal{#1}}
  \nc{\om}{\omega}
\nc{\gl}{\mathfrak{gl}}
  \nc{\cF}{\mathcal{F}}
\nc{\cC}{\mathcal{C}}
  \nc{\Mor}{\mathsf{Mor}}
  \nc{\HOM}{\operatorname{HOM}}
  \nc{\Ob}{\mathsf{Ob}}
  \nc{\Vect}{\mathsf{Vect}}
\nc{\gVect}{\mathsf{gVect}}
  \nc{\modu}{\mathsf{-mod}}
\nc{\pmodu}{\mathsf{-pmod}}
  \nc{\qvw}[1]{\La(#1 \Bv,\Bw)}
  \nc{\van}[1]{\nu_{#1}}
  \nc{\Rperp}{R^\vee(X_0)^{\perp}}
  \nc{\si}{\sigma}
\nc{\sgns}{{\boldsymbol{\sigma}}}
  \nc{\croot}[1]{\al^\vee_{#1}}
\nc{\di}{\mathbf{d}}
  \nc{\SL}[1]{\mathrm{SL}_{#1}}
  \nc{\Th}{\theta}
  \nc{\vp}{\varphi}
  \nc{\wt}{\mathrm{wt}}
\nc{\te}{\tilde{e}}
\nc{\tf}{\tilde{f}}
\nc{\hwo}{\mathbb{V}}
\nc{\soc}{\operatorname{soc}}
\nc{\cosoc}{\operatorname{cosoc}}
 \nc{\Q}{\mathbb{Q}}
\nc{\LPC}{\mathsf{LPC}}
  \nc{\Z}{\mathbb{Z}}
  \nc{\Znn}{\Z_{\geq 0}}
  \nc{\ver}{\EuScript{V}}
  \nc{\Res}{\operatorname{Res}}
  \nc{\simples}{X}
  \nc{\edge}{\EuScript{E}}
  \nc{\Spec}{\mathrm{Spec}}
  \nc{\tie}{\EuScript{T}}
  \nc{\ml}[1]{\mathbb{D}^{#1}}
  \nc{\fQ}{\mathfrak{Q}}
        \nc{\fg}{\mathfrak{g}}
        \nc{\ft}{\mathfrak{t}}
  \nc{\Uq}{U_q(\fg)}
        \nc{\bom}{\boldsymbol{\omega}}
\nc{\bla}{{\underline{\boldsymbol{\la}}}}
\nc{\bmu}{{\underline{\boldsymbol{\mu}}}}
\nc{\bal}{{\boldsymbol{\al}}}
\nc{\bet}{{\boldsymbol{\eta}}}
\nc{\rola}{X}
\nc{\wela}{Y}
\nc{\fM}{\mathfrak{M}}
\nc{\fX}{\mathfrak{X}}
\nc{\fH}{\mathfrak{H}}
\nc{\fE}{\mathfrak{E}}
\nc{\fF}{\mathfrak{F}}
\nc{\fI}{\mathfrak{I}}
\nc{\qui}[2]{\fM_{#1}^{#2}}
\nc{\cL}{\mathcal{L}}
\nc{\ca}[2]{\fQ_{#1}^{#2}}
\nc{\cat}{\mathcal{V}}
\nc{\cata}{\mathfrak{V}}
\nc{\catf}{\mathscr{V}}
\nc{\hl}{\mathcal{X}}
\nc{\hld}{\EuScript{X}}
\nc{\hldbK}{\EuScript{X}^{\bla}_{\bar{\mathbb{K}}}}
\nc{\Iwahori}{\EuScript{I}}
\nc{\WC}{\EuScript{C}}

\nc{\pil}{{\boldsymbol{\pi}}^L}
\nc{\pir}{{\boldsymbol{\pi}}^R}
\nc{\cO}{\mathcal{O}}
\nc{\Ko}{\text{\Denarius}}
\nc{\Ei}{\fE_i}
\nc{\Fi}{\fF_i}
\nc{\fil}{\mathcal{H}}
\nc{\brr}[2]{\beta^R_{#1,#2}}
\nc{\brl}[2]{\beta^L_{#1,#2}}
\nc{\so}[2]{\EuScript{Q}^{#1}_{#2}}
\nc{\EW}{\mathbf{W}}
\nc{\rma}[2]{\mathbf{R}_{#1,#2}}
\nc{\Dif}{\EuScript{D}}\nc{\MDif}{\EuScript{E}}
\renc{\mod}{\mathsf{mod}}
\nc{\modg}{\mathsf{mod}^g}
\nc{\fmod}{\mathsf{mod}^{fd}}
\nc{\id}{\operatorname{id}}
\nc{\compat}{\EuScript{K}}
\nc{\DR}{\mathbf{DR}}
\nc{\End}{\operatorname{End}}
\nc{\Fun}{\operatorname{Fun}}
\nc{\Ext}{\operatorname{Ext}}
\nc{\tw}{\tau}
\nc{\A}{\EuScript{A}}
\nc{\Loc}{\mathsf{Loc}}
\nc{\eF}{\EuScript{F}}
\nc{\LAA}{\Loc^{\A}_{A}}
\nc{\gfq}[2]{B_{#1}^{#2}}
\nc{\qgf}[1]{A_{#1}}
\nc{\qgr}{\qgf\rho}
\nc{\tqgf}{\tilde A}
\nc{\Tr}{\operatorname{Tr}}
\nc{\Tor}{\operatorname{Tor}}
\nc{\cQ}{\mathcal{Q}}
\nc{\st}[1]{\Delta(#1)}
\nc{\cst}[1]{\nabla(#1)}
\nc{\ei}{\mathbf{e}_i}
\nc{\Be}{\mathbf{e}}
\nc{\Hck}{\mathfrak{H}}
\renc{\P}{\mathbb{P}}
\nc{\bbB}{\mathbb{B}}
\nc{\ssy}{\mathsf{y}}
\nc{\cI}{\mathcal{I}}
\nc{\cG}{\mathcal{G}}
\nc{\cH}{\mathcal{H}}
\nc{\coe}{\mathfrak{K}}
\nc{\pr}{\operatorname{pr}}
\nc{\bra}{\mathfrak{B}}
\nc{\rcl}{\rho^\vee(\la)}
\nc{\tU}{\mathcal{U}}
\nc{\dU}{{\stackon[8pt]{\tU}{\cdot}}}
\nc{\dT}{{\stackon[8pt]{\cT}{\cdot}}}

\nc{\RHom}{\R\operatorname{Hom}}
\nc{\tcO}{\tilde{\cO}}
\nc{\Yon}{\mathscr{Y}}
\nc{\sI}{{\mathsf{I}}}
\nc{\sptc}{\ft_{1,\Z}}
\nc{\spt}{\ft_{1,\R}}
\nc{\Bpsi}{u}
\nc{\rankm}{m-1}
\nc{\rankp}{m}
\nc{\acham}{\eta}
\nc{\hyper}{\mathsf{H}}
\nc{\AF}{\EuScript{F}\ell}
\nc{\VB}{\EuScript{X}}
\nc{\OHiggs}{\cO_{\operatorname{Higgs}}}
\nc{\OCoulomb}{\cO_{\operatorname{Coulomb}}}
\nc{\tOHiggs}{\tilde\cO_{\operatorname{Higgs}}}
\nc{\tOCoulomb}{\tilde\cO_{\operatorname{Coulomb}}}
\nc{\indx}{\mathcal{I}}
\nc{\redu}{K}
\nc{\Ba}{\mathbf{a}}
\nc{\Bb}{\mathbf{b}}
\nc{\Lotimes}{\overset{L}{\otimes}}
\nc{\AC}{C}
\nc{\ideal}{\mathscr{I}}
\nc{\ACs}{\mathscr{C}}
\nc{\Stein}{\mathscr{X}}
\newcommand{\cOg}{\mathcal{O}_{\!\operatorname{g}}}
\newcommand{\tcOg}{\mathcal{\tilde O}_{\!\operatorname{g}}}
\newcommand{\dOg}{D_{\cOg}}
\newcommand{\preO}{p\cOg}
\newcommand{\dpreO}{D_{p\cOg}}
\nc{\No}{H}
\nc{\To}{Q}
\nc{\tNo}{\tilde{H}}
\nc{\tTo}{\tilde{Q}}
\nc{\flav}{\phi}
\nc{\tF}{\tilde{F}}
\nc{\auto}{\alpha}
\nc{\zetal}{\zeta}
\nc{\zetap}{\beta}
\nc{\DO}{\mathsf{U}}
\nc{\Wei}{\EuScript{W}}
\nc{\bQ}{\mathbf{Q}}
\nc{\lp}{\nicefrac{\ell}{p}}
\newcommand{\CC}{\mathbb{C}}
\newcommand{\RR}{\mathbb{R}}
\newcommand{\ZZ}{\mathbb{Z}}
\newcommand{\cM}{\mathcal{M}}
\newcommand{\cN}{\mathcal{N}}
\newcommand{\sK}{\mathscr{K}}
\newcommand{\sL}{\mathscr{L}}
\newcommand{\sM}{\mathscr{M}}
\newcommand{\sP}{\mathscr{P}}
\newcommand{\sU}{\mathscr{U}}
\newcommand{\T}{\mathbf{T}}
\renewcommand{\U}{\mathbf{U}}
\nc{\vT}{\mathbb{T}}
\nc{\bvT}{T}

\nc{\tHom}{\underline{\operatorname{Hom}}}
\nc{\sHom}{\underline{\mathsf{Hom}}}

\nc{\wgmod}{\operatorname{-wgmod}}
\newcommand{\diag}{\mathrm{diag}}
\newcommand{\Frac}{\mathrm{Frac}}
\nc{\tWei}{\widetilde{\EuScript{W}}}
\setcounter{tocdepth}{2}
\newcommand{\thetitle}{Three perspectives on\\ categorical symmetric Howe duality}
\newcommand{\theshorttitle}{Three perspectives on categorical symmetric Howe duality}
\nc{\Fl}{\operatorname{Fl}}
\nc{\Pervc}[1]{\mathcal{P}erv_{#1}}
\nc{\mGT}{\mathsf{m}}
\nc{\perv}{\mathscr{F}}
\nc{\MaxSpec}{\operatorname{MaxSpec}}
\renc{\theitheorem}{\Alph{itheorem}}
\newcommand{\btodo}{\textsf}
\newcommand{\etodo}{\todo[inline,color=red!20]}

\baselineskip=1.1\baselineskip

 \usetikzlibrary{decorations.pathreplacing,backgrounds,decorations.markings,shapes.geometric,fit}
 \tikzset{tstar/.style={fill=white,draw,star,star points=5,star point ratio=0.45,inner
		sep=3pt}}
\tikzset{ucircle/.style={fill=black,circle,inner sep=2pt}}

\tikzset{wei/.style={draw=red,double=red!40!white,double distance=1.5pt,thin}}
\tikzset{awei/.style={draw=blue,double=blue!40!white,double distance=1.5pt,thin}}
\tikzset{bdot/.style={fill,circle,color=blue,inner sep=3pt,outer
    sep=0}}
\tikzset{dir/.style={postaction={decorate,decoration={markings,
    mark=at position .8 with {\arrow[scale=1.3]{>}}}}}}
\tikzset{rdir/.style={postaction={decorate,decoration={markings,
    mark=at position .8 with {\arrow[scale=1.3]{<}}}}}}
\tikzset{edir/.style={postaction={decorate,decoration={markings,
    mark=at position .2 with {\arrow[scale=1.3]{<}}}}}}

\tikzset{mid/.style={postaction={decorate,decoration={markings,
    mark=at position .5 with {\arrow[scale=1.5]{>}}}}}}  
    
    \begin{center}
\noindent {\large  \bf \thetitle}\\
\bigskip

\author[1,2]{ Ben Webster}
 \@author\\\bigskip

\author[1]{Jerry Guan}
\affil[1]{Department of Pure Mathematics,
   University of Waterloo \&} 
\affil[2]{Perimeter Institute for Mathematical Physics } 
{\rm with an appendix by }\\ 
\@author\\
{\textsl Waterloo, ON, Canada }\\
{\tt ben.webster@uwaterloo.ca}
\end{center}
\bigskip
{\small
\begin{quote}
\noindent {\em Abstract.}
In this paper, we consider the categorical symmetric Howe duality introduced by Khovanov, Lauda, Sussan and Yonezawa.  While originally defined from a purely diagrammatic perspective, this construction also has geometric and representation-theoretic interpretations, corresponding to certain perverse sheaves on spaces of quiver representations and the category of Gelfand-Tsetlin modules over $\mathfrak{gl}_n$.  

In particular, we show that the ``deformed Webster algebras'' discussed in \cite{KLSY} manifest a Koszul duality between blocks of the category of Gelfand-Tsetlin modules over $\mathfrak{gl}_n$, and the constructible sheaves on representations of a linear quiver invariant under a certain parabolic in the group that acts by changing bases.  Furthermore, we show that this duality intertwines translation functors with a diagrammatic categorical action (generalizing that of  \cite{KLSY}). 
\end{quote}
}
\section{Introduction}
In this paper, we discuss 3 different perspectives on a category which has shown up in several guises in recent years.  We can organize these perspectives as (1) diagrammatic, (2) representation-theoretic and (3) geometric. 

For the diagrammatic perspective, work of Khovanov and Lauda \cite{KLI} initiated a great burst of different algebras spanned by diagrams with locally defined relations; while no rubric can contain all of this efflorescence, the author introduced {\bf weighted Khovanov-Lauda-Rouquier algebras} \cite{WebwKLR} which include the algebras discussed in this paper (and many others we will not consider here).  The special case of interest to us was considered in recent work of Khovanov-Lauda-Sussan-Yonezawa \cite{KLSY};  in their terminology, this is a {\bf deformed Webster algebra}.  In the spirit of compromise, we will follow the terminology suggested by our collaborators in \cite{KTWWYO}, and write {\bf KLRW algebras}.  

As suggested by the title, we will focus on the case which is relevant to symmetric Howe duality; specifically, we consider the algebras that categorify $\mathfrak{sl}_\infty$-weight spaces of the symmetric power  $U(\mathfrak{n})\otimes \Sym^n(\C^{\infty}\otimes  \C^{\rankp})$, where $\mathfrak{n}\subset \mathfrak{sl}_{\rankp}$ is the Lie algebra of strictly upper-triangular matrices.   We can identify the $\mathfrak{sl}_\infty$ weights appearing with increasing $n$-tuples $\chi=(\chi_1\leq \chi_2\leq \cdots \leq \chi_n)$, and we denote the algebra categorifying this weight space by $\bT^\chi$. One can easily verify that as an $\mathfrak{sl}_m$-module, this weight space can be identified with $U(\mathfrak{n})\otimes \Sym^{g_1}(\C^m)\otimes \cdots \otimes \Sym^{g_k}(\C^m)$, where $g_*$ are multiplicities with which the distinct elements of $\chi$ repeat.  The algebra $\bT^\chi$ is a deformed version of the algebra $\tilde{T}^\bla$ as introduced in \cite[\S 4]{Webmerged} which is shown there to give a categorification of this tensor product. In  \cite{KLSY}, Khovanov-Lauda-Sussan-Yonezawa consider the case $\rankp=2$ and construct the Howe dual categorical $\mathfrak{sl}_\infty$-action on the categories of $\bT^\chi$-modules.  In this paper, we will generalize this result to all values of $\rankp$.  

We achieve this by considering the other perspectives mentioned above, where the Howe dual actions are consequences of previously constructed $\mathfrak{sl}_\infty$-actions.  
From the representation-theoretic perspective, we consider the category of {\bf Gelfand-Tsetlin modules} over $\mathfrak{gl}_n$.  Recall that a Gelfand-Tsetlin module over $\mathfrak{gl}_n$ is one on which the center of $U(\mathfrak{gl}_k)$ for all $k\leq n$ acts locally finitely; we will also sometimes want to consider  pro-Gelfand-Tsetlin modules, by which we mean topological modules where the action is only topologically locally finite.  This category has received a great deal of interest in recent years \cite{FGRnew,RZ18, FGRZVerma} but its objects have remained relatively mysterious.  Recent work of the author and collaborators \cite{KTWWYO,WebGT} gave a classification of the simples in a block of this category in general, but the combinatorics of the general case is somewhat complicated; some data on the complexity of the $\mathfrak{sl}_3$ and $\mathfrak{sl}_4$ case are presented in \cite{SilverthorneW}.  One of our motivations in this paper is to draw out the structure of this category in the most interesting case, that of an integral central character.
We'll show here (based on the techniques in \cite{WebGT}) that the algebras $\bT^\chi$ attached to the zero weight space for the action of $\mathfrak{sl}_{\rankp}$ control the category  $\mathcal{GT}_\chi$ of integral Gelfand-Tsetlin modules where now we interpret $\chi$ as a  central character of $Z_n=Z(U(\mathfrak{gl}_n))$.  We can also associate a parabolic $P_\chi$ to $\chi$, whose block sizes are the multiplicities of entries in $\chi$ which coincide.  

The algebras $\bT^\chi$ also have a topological interpretation in terms of convolution algebras and perverse sheaves.   We can view this as a
generalization of the well-known theorem of Beilinson-Ginzburg-Soergel which shows that the Koszul dual of an integral block $\mathcal{O}_\chi$ of deformed category $\mathcal{O}$ is the category of $P$-equivariant perverse sheaves on $GL_n/B$ where $P=P_\chi$ corresponds to the central character $\chi$ of the block.  Our main theorem explains how to extend this to Gelfand-Tsetlin modules.

Consider the vector space $V$ defined by  the set of quiver representations on the vector spaces $\C^1\overset{f_1}\to \C^2\overset{f_2}\to \cdots \overset{f_{n-2}}\to\C^{n-1}\overset{f_{n-1}}\to \C^{n}$ divided by the group $G$ that changes bases arbitrarily on $\C^1,\dots, \C^{n-1}$ and on $\C^{n}$ by elements of the group $P\subset GL_n$ (that is, preserving the standard partial flag corresponding to $P$). That is, \newseq
\begin{align*}
\label{eq:OGZ1}\subeqn G&=GL_{{1}}\times \cdots \times GL_{n-1}\times P\\
\label{eq:OGZ2}\subeqn  V&=\Hom(\C^1,\C^2)\oplus \cdots \oplus
      \Hom(\C^{n-2},\C^{n-1})\oplus \Hom(\C^{n-1},\C^n).
\end{align*}
with $\Gz=GL_{{n-1}}\times \cdots \times GL_{1}$.
Note that on an open subset of $V$, the maps $f_i$ are all injective, and the subspaces $$\operatorname{im}(f_{n-1}\cdots f_1)\subset\cdots \subset\operatorname{im}(f_{n-1}f_{n-2})\subset\operatorname{im}(f_{n-1})\subset \C^n$$ give a complete flag.  Thus, we can identify this open subset of $V/\Gz$ with the flag variety $\Fl=GL_n/B$.  In this paper, we will study $G$-equivariant sheaves on $V$ as an enlargement of the category of $P$-equivariant sheaves on $\Fl$.  

We let $\Gc{\chi}=P_{\chi}\times \Gz$ and consider the usual $\Gc{\chi}$-equivariant derived category of $\C$-vector spaces $D^{b}_{\Gc{\chi}}(V)$; we'll show that this category has a graded lift $D^{b,\operatorname{mix}}_{\Gc{\chi}}(V)$.   
\begin{itheorem}\label{thm:main}
We have equivalences of categories between:
\begin{enumerate}
    \item The category of weakly gradable finite-dimensional $\bT^\chi$-modules.
    \item The category $\mathcal{GT}_\chi$ of integral Gelfand-Tsetlin modules.
\end{enumerate}
Thus, the category $\bT^\chi\operatorname{-gmod}$ of all finitely generated graded $\bT^\chi$-modules is a graded lift $\widetilde{\mathcal{GT}}_\chi$ of the category of pro-Gelfand-Tsetlin modules.

We also have an equivalence of categories  between:
\begin{enumerate}
    \item[(1')] The category of linear complexes of projectives over $\bT^\chi$.  
    \item[(3)] The category of $P_\chi$-equivariant perverse sheaves on $V$.
\end{enumerate}
The categories $(1)$ and $(1')$ are in a certain sense Koszul dual, so the same is true of $(2)$ and $(3)$.  These equivalences are induced by equivalences of derived categories
\[D^b(\widetilde{\mathcal{GT}}_\chi) \cong D^b(\bT^\chi\operatorname{-gmod})\cong  D^{b,\operatorname{mix}}_{\Gc{\chi}}(V).\]
\end{itheorem}
Just as with the original Koszul duality of Beilinson--Ginzburg--Soergel, this result seems to be a manifestation of the self-duality of $T^*GL_n/B$ under 3-dimensional mirror symmetry \cite{webster3DimensionalMirror2023}.  In particular, the Koszul duality of (2) and (3) is a special case of \cite[Prop. 4.7]{WebSD}, which relates versions of these categories for arbitrary Higgs and Coulomb branches of 3-dimensional gauge theories.

Furthermore, this equivalence matches two natural actions of 2-categories on the categories appearing in the theorem above.  As discussed previously, we can interpret $\chi$ as a weight of $\mathfrak{sl}_\infty$, and $K^0(\bT^\chi )$ as a weight space of a $\mathfrak{sl}_\infty$-module.  Thus, it's a natural question whether this can be extended to a categorical $\mathfrak{sl}_{\infty}$-action.  Not only is this possible, but in fact, the resulting action is one already known in both the representation-theoretic and geometric perspectives.

For Gelfand-Tsetlin modules, this action is by translation functors.  The functors of $\mathsf{E}(M)=\C^n\otimes M$ and $\mathsf{F}(M)=(\C^n)^*\otimes M$ act on the category of Gelfand-Tsetlin modules, and define an action of the level 0 Heisenberg category (also called the affine oriented Brauer category in \cite{BCNR}) on the category of all Gelfand-Tsetlin modules.  These functors decompose according to how they act on blocks, and by \cite[Th. A]{BSW3}, the summands of this functor define a categorical $\mathfrak{sl}_\infty$ action on the sum of the integral blocks $\mathcal{GT}_\chi$.  

On the other hand, as is always true for equivariant sheaves for different subgroups of a single group, the  $\Gc{\chi}=P_\chi\times \Gz$ equivariant derived categories of $V$ for different $\chi$ carry an action by convolution of the derived categories $D^b_{P_{\chi'}\times P_\chi} ( GL_n)$ where these subgroups act by left and right multiplication.   This is essentially an action of the category $\mathsf{Perv}$ from \cite[Def. 5]{Webcomparison} with minor notational changes to account for working with $\mathfrak{sl}_\infty.$  There is a 2-functor $\Phi$ from the 2-Kac-Moody algebra for $\mathfrak{sl}_\infty$ to $\mathsf{Perv}$ introduced in \cite{Webcomparison}, uniquely characterized by the property that it agrees with Khovanov and Lauda's original functor from \cite{KLIII} to modules over the cohomology rings of $GL_n/P_\chi$.

Both these actions must have an algebraic description in terms of bimodules over $\bT^{\chi'}$ and $\bT^{\chi}$.  In fact, the resulting bimodules are the {\bf ladder bimodules} defined in \cite{KLSY} in the case of $\mathfrak{sl}_2$.  Our comparison of these with the geometric action gives an easy and conceptual proof of the fact that these bimodules induce a categorical $\mathfrak{sl}_\infty$ action in the case not just of $\mathfrak{sl}_2$ (as is shown in \cite{KLSY}), but the more general case of $\mathfrak{sl}_n$.  
\begin{itheorem}\label{thm:action}
  The equivalences of Theorem \ref{thm:main} match:
  \begin{enumerate}
      \item The categorical $\mathfrak{sl}_\infty$-action on $\oplus_\chi \bT^\chi\mmod$ defined by ladder bimodules as in \cite{KLSY}.
      \item The categorical $\mathfrak{sl}_\infty$-action on $\oplus_\chi \mathcal{GT}_\chi$ defined by translation functors.
      \item The categorical $\mathfrak{sl}_\infty$-action on $\oplus_\chi D^{b,\operatorname{mix}}_{\Gc{\chi}}(V)$ defined by convolution with sheaves in $\mathsf{Perv}$.
  \end{enumerate}
\end{itheorem}
The actions (1) and (3) make sense when we replace \eqref{eq:OGZ1} and \eqref{eq:OGZ2} with more general dimension vectors, as we'll discuss below, but (2) really depends on the identification with the universal enveloping algebra.  
\excise{arbitrary dimension vectors, that is for 
\[V=\Hom(\C^{v_1},\C^{v_2})\oplus \Hom(\C^{v_2},\C^{v_3})\oplus \cdots \Hom(\C^{v_{m-1}},\C^{v_{m}})\oplus \Hom(\C^{v_k},\C^n)\] for arbitrary $1\leq k\leq m$ and $(v_1,\dots, v_{m})$.  This corresponds to the quiver gauge theory 
\begin{equation*}
\tikz{
	\node[draw, thick, circle, inner sep=5pt,fill=white] (a) at (-6,0) {$v_1$};
	\node[draw, thick, circle, inner sep=5pt,fill=white] (b) at (-4,0) {$v_2$};	
		\node[draw, thick, circle, inner sep=5pt,fill=white] (d) at (0,0) {$v_k$};
		\node[draw, thick, circle, inner sep=5pt,fill=white] (f) at (4,0) {$v_m$};
	\node[draw, thick,inner sep=8pt] (s) at (0,2){$n$};
	\node[inner sep=12pt,fill=white] (c) at (-2,0){$\cdots$};
	\node[inner sep=12pt,fill=white] (e) at (2,0){$\cdots$};
		\draw[thick] (a) -- (b) ;
		\draw[thick] (b) -- (c) ;
		\draw[thick] (c) -- (d) ;
		\draw[thick] (d) -- (e) ;
		\draw[thick] (e) -- (f) ;
		\draw[thick] (d) -- (s) ;
}
\end{equation*}}
Theorems \ref{thm:main} and \ref{thm:action} extend essentially without change to the comparison of the equivariant derived category and $\bT$, and the match of the categorical $\mathfrak{sl}_\infty$-actions.  
The extension of the action (2) will require more effort, though it should be possible in some cases where the quantum Coulomb branch is a finite W-algebra using Brundan and Kleshchev's definition of translation functors for W-algebras in \cite[\S 4.4]{BKrsy}.

Finally, we discuss the slightly tangled relationship of this construction to previous work relating diagrammatic categories and the representation theory of Lie algebras.  

 \subsection*{Acknowledgments}
 J. G. was supported by NSERC and the University of Waterloo through an Undergraduate Student Research Award.  B. W. is supported by an NSERC Discovery Grant.   This research was supported in part by Perimeter
Institute for Theoretical Physics.
Research at Perimeter Institute is supported in part by the Government of Canada through the Department of Innovation, Science and Economic Development Canada and by the Province of Ontario through the Ministry of Colleges and Universities.
\section{Diagrammatic algebras}
\subsection{Algebras}

In this section, we introduce the algebraic construction that will unify the perspectives appearing here.  These algebras have appeared in many previous works of the author and collaborators \cite{WebCB, WebTGK,Webmerged,WebwKLR,KTWWYO,WebGT}, as well as the line of research of Khovanov, Lauda, Sussan and Yonezawa \cite{KSred,KLSY}:
\begin{definition}
  A {\bf KLRW diagram} is a collection of finitely many oriented curves in
  $\R\times [0,1]$ whose projection to the second factor is a diffeomorphism. Each curve is either
  \begin{itemize}
  \item colored red and labeled with the integer $\rankp$ and decorated with finitely many dots;  
  \item colored black and labeled with $i\in [1,\rankm]$ and decorated with finitely many dots.  Let $v_i$ be the number of black strands with label $i$.
  \end{itemize}
 The diagram must be locally of the form 
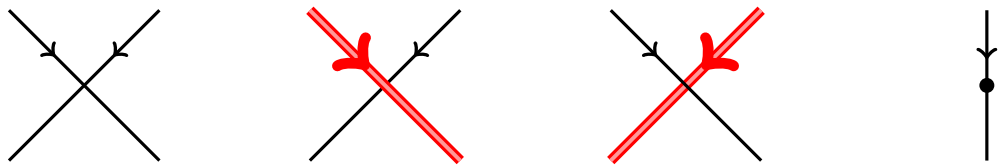
\begin{equation*}
\begin{tikzpicture}[scale=.8]
\draw[very thick,postaction={decorate,decoration={markings,
    mark=at position .75 with {\arrow[scale=1.3]{<}}}}] (-4,0) +(-1,-1) -- +(1,1);
\draw[very thick,postaction={decorate,decoration={markings,
    mark=at position .75 with {\arrow[scale=1.3]{<}}}}](-4,0) +(1,-1) -- +(-1,1);

  \draw[very thick,postaction={decorate,decoration={markings,
    mark=at position .75 with {\arrow[scale=1.3]{<}}}}](0,0) +(-1,-1) -- +(1,1);
\draw[wei, very thick,postaction={decorate,decoration={markings,
    mark=at position .75 with {\arrow[scale=1.3]{<}}}}](0,0) +(1,-1) -- +(-1,1);

  \draw[wei,very thick,postaction={decorate,decoration={markings,
    mark=at position .75 with {\arrow[scale=1.3]{<}}}}](4,0) +(-1,-1) -- +(1,1);
\draw [very thick,postaction={decorate,decoration={markings,
    mark=at position .75 with {\arrow[scale=1.3]{<}}}}](4,0) +(1,-1) -- +(-1,1);

  \draw[very thick,postaction={decorate,decoration={markings,
    mark=at position .75 with {\arrow[scale=1.3]{<}}}}](7.5,0) +(0,-1) --  node
  [midway,circle,fill=black,inner sep=2pt]{}
  +(0,1);
    \draw[wei,postaction={decorate,decoration={markings,
    mark=at position .75 with {\arrow[scale=1.3]{<}}}}](10,0) +(0,-1) --  node
  [midway,circle,fill=red,inner sep=2pt]{}
  +(0,1);
\end{tikzpicture}
\end{equation*}
with each curve oriented in the negative direction.  In
particular, no red strands can ever cross.  Each curve must
meet both $y=0$ and $y=1$ at distinct points from the other curves.
\end{definition} 
Readers familiar with the conventions of \cite{WebTGK,Webmerged}, etc. might be surprised to see dots on red strands as well as black.  This corresponds to the ``canonical deformation'' discussed in \cite[\S 2.7]{WebwKLR} or the ``redotting'' of \cite{KSred}.  

We'll only consider KLRW diagrams up to isotopy. Since
the orientation on a diagram is clear, we typically won't draw it.

We call the lines $y=0,1$ the {\bf  bottom} and {\bf top} of the
diagram.  Reading across the bottom and top from left to right, we
obtain a sequence $\Bi=(i_1, \dots, i_V)$ of elements of $[1,\rankp]$ labelling both red and black strands, where $V$ is the total number of strands.

\begin{definition}
  Given KLRW diagrams $a$ and $b$, their {\bf composition} $ab$ is
  given by stacking $a$ on top of $b$ and attempting to join the
  bottom of $a$ and top of $b$. If the sequences
  from the bottom of $a$ and top of $b$ don't match, then the
  composition is not defined and by convention is 0, which is not a
  KLRW diagram, just a formal symbol.
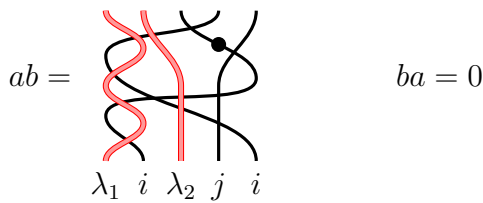
\[
ab=
\begin{tikzpicture}[baseline,very thick,yscale=.5]
  \draw (-.5,-2) to[out=90,in=-90] node[below,at start]{$i$} (-1,-.8)
  to[out=90,in=-90](1,.2) to[out=90,in=-90] node[midway,circle,fill=black,inner sep=2pt]{}  (0,2);
  \draw (.5,-2) to[out=90,in=-90] node[below,at start]{$j$} (.5,0) to[out=90,in=-90] (1,2);
  \draw  (1,-2) to[out=90,in=-90] node[below,at start]{$i$} (-1,.8) to[out=90,in=-90] (.5,2);
  \draw[wei] (0,-2) to[out=90,in=-90] node[below,at start]{$\la_2$}
  (0,0) to[out=90,in=-90]
  (-.5,2);
  \draw[wei] (-1, -2) to[out=90,in=-90] node[below,at start]{$\la_1$}
  (-.5,-1) to[out=90,in=-90] (-1,0) to[out=90,in=-90] (-.5,1)
  to[out=90,in=-90]   (-1,2);
\end{tikzpicture}\qquad \qquad ba=0
\]
Fix a field $\K$ and let $\ttalg$ be the formal span over
$\K$ of KLRW diagrams (up to isotopy).  The composition law
induces an algebra structure on $\ttalg$.
\end{definition}
 Let $e(\Bi)$ be the unique
crossingless, dotless  diagram where the sequence at top and bottom are both $\Bi$.    

\begin{definition}  The {\bf degree} of a KLRW diagram is the sum over
  crossings and dots in the diagram of 
  \begin{itemize}
  \item$-\langle\al_i,\al_j\rangle$ for each crossing of a black strand
    labeled $i$ with one labeled $j$;
  \item $2$ for each dot on a red strand or a black
    strand;
  \item $\langle\al_i,\la\rangle=\la^i$ for each crossing of a
    black strand labeled $i$ with a red strand labeled $\la$.
  \end{itemize}
The degree of diagrams is additive under composition.  Thus, the
algebra $\ttalg$ inherits a grading from this degree function.
\end{definition}

Throughout, we fix integers $\rankp,n$ and $\chi$ an integral weight $(\chi_1,\dots, \chi_n)\in \Z^n$ such that $\chi_1\leq \cdots \leq \chi_n$.  \notation{$\chi$}{A weakly increasing sequence }Fix a dimension vector $\Bv=(v_1,\dots, v_{\rankp-1})\in \Z_{\geq 0}^{\rankp-1}$; by convention, we take $v_{\rankp}=n$. Consider the set $\Omega=\{(i,j) \mid i\in [1,\rankp], j\in [1,v_i]\}$.\notation{$\Omega$}{The indexing set $\Omega=\{(i,j) \mid i\in [1,\rankp], j\in [1,v_i]\}$ for bases of $\C^{v_i}$ for all $i$ or for black strands in a KLRW diagram.}
Let $\prec$ be a total order on $\Omega$ such that 
\begin{equation}
    (i,1)\preceq \cdots \preceq (i,v_i). \label{eq:order-1}
\end{equation}  
This is equivalent to choosing a word $\Bi =(i_1,\dots, i_{N})$ where $N=|\Omega|$ and $i_k=i$ for $v_i$ different indices $k$.  

We will want to weaken this definition a bit and allow $\preceq$ to be a total preorder (that is, a relation which is transitive and reflexive, but not necessarily anti-symmetric).  In this case, we have an induced equivalence relation $(i,k)\approx (j,\ell)$ if $(i,k)\preceq (j,\ell)$ and $(i,k)\succeq (j,\ell)$.  We assume that our preorder satisfies the condition that
\begin{equation}
    (i,k)\not\approx (j,\ell)\text{ whenever } i\neq j.\label{eq:order-2}
\end{equation}
We can still attach a word $\Bi$ to such a preorder; two equivalent elements give the same letter in the word $\Bi$, so it doesn't matter whether they have a chosen order.  We can thus think of a preorder as corresponding to a word in the generators with some subsets where the same letter appears multiple times together grouped together.  We can represent this within the word itself by replacing $(i,\dots, i)$ with $i^{(a)}$.  Thus, for our purposes $(3,2,2,3,1,3)$ and $(3,2^{(2)},3,1,3)$ are different words with different associated preorders.  Every such word has a unique {\bf totalization} satisfying \eqref{eq:order-1}.  
\begin{definition}
We say the preorder $\prec$ and word  $\Bi$ is $\chi$-parabolic if whenever $\chi_k=\chi_{k+1}$ then $(\rankp,k)\approx (\rankp,k+1)$. In particular, the corresponding appearances of $\rankp$  in $\Bi$ are consecutive.
\end{definition}

\begin{definition}
  Let $\bT$ be the quotient  of $\ttalg$ by
  the following local
  relations between KLRW diagrams.  We draw these below as black, but the same relations apply to red strands (always taken with the label $\rankp$):
\begin{equation*}\subeqn\label{first-QH}
    \begin{tikzpicture}[scale=.6,baseline]
      \draw[very thick,postaction={decorate,decoration={markings,
    mark=at position .2 with {\arrow[scale=1.3]{<}}}}](-4,0) +(-1,-1) -- +(1,1) node[below,at start]
      {$i$}; \draw[very thick,postaction={decorate,decoration={markings,
    mark=at position .2 with {\arrow[scale=1.3]{<}}}}](-4,0) +(1,-1) -- +(-1,1) node[below,at
      start] {$j$}; \fill (-4.5,.5) circle (3pt);
\node at (-2,0){=}; \draw[very thick,postaction={decorate,decoration={markings,
    mark=at position .8 with {\arrow[scale=1.3]{<}}}}](0,0) +(-1,-1) -- +(1,1)
      node[below,at start] {$i$}; \draw[very thick,postaction={decorate,decoration={markings,
    mark=at position .8 with {\arrow[scale=1.3]{<}}}}](0,0) +(1,-1) --
      +(-1,1) node[below,at start] {$j$}; \fill (.5,-.5) circle (3pt);
   \end{tikzpicture}
 \qquad 
    \begin{tikzpicture}[scale=.6,baseline]
      \draw[very thick,postaction={decorate,decoration={markings,
    mark=at position .2 with {\arrow[scale=1.3]{<}}}}](-4,0) +(-1,-1) -- +(1,1) node[below,at start]
      {$i$}; \draw[very thick,postaction={decorate,decoration={markings,
    mark=at position .2 with {\arrow[scale=1.3]{<}}}}](-4,0) +(1,-1) -- +(-1,1) node[below,at
      start] {$j$}; \fill (-3.5,.5) circle (3pt);
\node at (-2,0){=}; \draw[very thick,postaction={decorate,decoration={markings,
    mark=at position .8 with {\arrow[scale=1.3]{<}}}}](0,0) +(-1,-1) -- +(1,1)
      node[below,at start] {$i$}; \draw[very thick,postaction={decorate,decoration={markings,
    mark=at position .8 with {\arrow[scale=1.3]{<}}}}](0,0) +(1,-1) --
      +(-1,1) node[below,at start] {$j$}; \fill (-.5,-.5) circle (3pt);
      \node at (3,0){unless $i=j$};
    \end{tikzpicture}
  \end{equation*}
\begin{equation*}\subeqn\label{nilHecke-1}
    \begin{tikzpicture}[scale=.6,baseline]
      \draw[very thick,postaction={decorate,decoration={markings,
    mark=at position .2 with {\arrow[scale=1.3]{<}}}}](-4,0) +(-1,-1) -- +(1,1) node[below,at start]
      {$i$}; \draw[very thick,postaction={decorate,decoration={markings,
    mark=at position .2 with {\arrow[scale=1.3]{<}}}}](-4,0) +(1,-1) -- +(-1,1) node[below,at
      start] {$i$}; \fill (-4.5,.5) circle (3pt);
\node at (-2,0){$-$}; \draw[very thick,postaction={decorate,decoration={markings,
    mark=at position .8 with {\arrow[scale=1.3]{<}}}}](0,0) +(-1,-1) -- +(1,1)
      node[below,at start] {$i$}; \draw[very thick,postaction={decorate,decoration={markings,
    mark=at position .8 with {\arrow[scale=1.3]{<}}}}](0,0) +(1,-1) --
      +(-1,1) node[below,at start] {$i$}; \fill (.5,-.5) circle (3pt);
      \node at (1.8,0){$=$}; 
    \end{tikzpicture}\,\,
\begin{tikzpicture}[scale=.6,baseline]
      \draw[very thick,postaction={decorate,decoration={markings,
    mark=at position .8 with {\arrow[scale=1.3]{<}}}}](-4,0) +(-1,-1) -- +(1,1) node[below,at start]
      {$i$}; \draw[very thick,postaction={decorate,decoration={markings,
    mark=at position .8 with {\arrow[scale=1.3]{<}}}}](-4,0) +(1,-1) -- +(-1,1) node[below,at
      start] {$i$}; \fill (-4.5,-.5) circle (3pt);
\node at (-2,0){$-$}; \draw[very thick,postaction={decorate,decoration={markings,
    mark=at position .2 with {\arrow[scale=1.3]{<}}}}](0,0) +(-1,-1) -- +(1,1)
      node[below,at start] {$i$}; \draw[very thick,postaction={decorate,decoration={markings,
    mark=at position .2 with {\arrow[scale=1.3]{<}}}}](0,0) +(1,-1) --
      +(-1,1) node[below,at start] {$i$}; \fill (.5,.5) circle (3pt);
      \node at (2,0){$=$}; \draw[very thick,postaction={decorate,decoration={markings,
    mark=at position .5 with {\arrow[scale=1.3]{<}}}}](4,0) +(-1,-1) -- +(-1,1)
      node[below,at start] {$i$}; \draw[very thick,postaction={decorate,decoration={markings,
    mark=at position .5 with {\arrow[scale=1.3]{<}}}}](4,0) +(0,-1) --
      +(0,1) node[below,at start] {$i$};
    \end{tikzpicture}
  \end{equation*}
\begin{equation*}\subeqn\label{black-bigon}
    \begin{tikzpicture}[very thick,scale=.8,baseline]
      \draw[postaction={decorate,decoration={markings,
    mark=at position .55 with {\arrow[scale=1.3]{<}}}}] (-2.8,0) +(0,-1) .. controls (-1.2,0) ..  +(0,1)
      node[below,at start]{$i$}; \draw[postaction={decorate,decoration={markings,
    mark=at position .55 with {\arrow[scale=1.3]{<}}}}] (-1.2,0) +(0,-1) .. controls
      (-2.8,0) ..  +(0,1) node[below,at start]{$j$}; 
   \end{tikzpicture}=\quad
   \begin{cases}
0 & i=j\\
     \begin{tikzpicture}[very thick,yscale=.6,xscale=.8,baseline=-3pt]
       \draw[postaction={decorate,decoration={markings,
    mark=at position .5 with {\arrow[scale=1.3]{<}}}}] (2,0) +(0,-1) -- +(0,1) node[below,at start]{$j$};
       \draw[postaction={decorate,decoration={markings,
    mark=at position .5 with {\arrow[scale=1.3]{<}}}}] (1,0) +(0,-1) -- +(0,1) node[below,at start]{$i$};
     \end{tikzpicture} & i\neq j,j\pm 1\\
   \begin{tikzpicture}[very thick,yscale=.6,xscale=.8,baseline=-3pt]
       \draw[postaction={decorate,decoration={markings,
    mark=at position .8 with {\arrow[scale=1.3]{<}}}}] (2,0) +(0,-1) -- +(0,1) node[below,at start]{$j$};
       \draw[postaction={decorate,decoration={markings,
    mark=at position .8 with {\arrow[scale=1.3]{<}}}}] (1,0) +(0,-1) -- +(0,1) node[below,at start]{$i$};\fill (2,0) circle (4pt);
     \end{tikzpicture}-\begin{tikzpicture}[very thick,yscale=.6,xscale=.8,baseline=-3pt]
       \draw[postaction={decorate,decoration={markings,
    mark=at position .8 with {\arrow[scale=1.3]{<}}}}] (2,0) +(0,-1) -- +(0,1) node[below,at start]{$j$};
       \draw[postaction={decorate,decoration={markings,
    mark=at position .8 with {\arrow[scale=1.3]{<}}}}] (1,0) +(0,-1) -- +(0,1) node[below,at start]{$i$};\fill (1,0) circle (4pt);
     \end{tikzpicture}& i=j-1\\
  \begin{tikzpicture}[very thick,baseline=-3pt,yscale=.6,xscale=.8]
       \draw[postaction={decorate,decoration={markings,
    mark=at position .8 with {\arrow[scale=1.3]{<}}}}] (2,0) +(0,-1) -- +(0,1) node[below,at start]{$j$};
       \draw[postaction={decorate,decoration={markings,
    mark=at position .8 with {\arrow[scale=1.3]{<}}}}] (1,0) +(0,-1) -- +(0,1) node[below,at start]{$i$};\fill (1,0) circle (4pt);
     \end{tikzpicture}-\begin{tikzpicture}[very thick,yscale=.6,xscale=.8,baseline=-3pt]
       \draw[postaction={decorate,decoration={markings,
    mark=at position .8 with {\arrow[scale=1.3]{<}}}}] (2,0) +(0,-1) -- +(0,1) node[below,at start]{$j$};
       \draw[postaction={decorate,decoration={markings,
    mark=at position .8 with {\arrow[scale=1.3]{<}}}}] (1,0) +(0,-1) -- +(0,1) node[below,at start]{$i$};\fill (2,0) circle (4pt);
     \end{tikzpicture}& i=j+1
   \end{cases}
  \end{equation*}
 \begin{equation*}\subeqn\label{triple-dumb}
    \begin{tikzpicture}[very thick,scale=.8,baseline=-3pt]
      \draw[postaction={decorate,decoration={markings,
    mark=at position .2 with {\arrow[scale=1.3]{<}}}}] (-2,0) +(1,-1) -- +(-1,1) node[below,at start]{$k$}; \draw[postaction={decorate,decoration={markings,
    mark=at position .8 with {\arrow[scale=1.3]{<}}}}]
      (-2,0) +(-1,-1) -- +(1,1) node[below,at start]{$i$}; \draw[postaction={decorate,decoration={markings,
    mark=at position .5 with {\arrow[scale=1.3]{<}}}}]
      (-2,0) +(0,-1) .. controls (-3,0) ..  +(0,1) node[below,at
      start]{$j$}; \node at (-.5,0) {$-$}; \draw[postaction={decorate,decoration={markings,
    mark=at position .8 with {\arrow[scale=1.3]{<}}}}] (1,0) +(1,-1) -- +(-1,1)
      node[below,at start]{$k$}; \draw[postaction={decorate,decoration={markings,
    mark=at position .2 with {\arrow[scale=1.3]{<}}}}] (1,0) +(-1,-1) -- +(1,1)
      node[below,at start]{$i$}; \draw[postaction={decorate,decoration={markings,
    mark=at position .5 with {\arrow[scale=1.3]{<}}}}] (1,0) +(0,-1) .. controls
      (2,0) ..  +(0,1) node[below,at start]{$j$}; \end{tikzpicture}=\quad
      \begin{cases} 
    \begin{tikzpicture}[very thick,yscale=.6,xscale=.8,baseline=-3pt]
     \draw[postaction={decorate,decoration={markings,
    mark=at position .5 with {\arrow[scale=1.3]{<}}}}] (6.2,0)
      +(1,-1) -- +(1,1) node[below,at start]{$k$}; \draw[postaction={decorate,decoration={markings,
    mark=at position .5 with {\arrow[scale=1.3]{<}}}}] (6.2,0)
      +(-1,-1) -- +(-1,1) node[below,at start]{$i$}; \draw[postaction={decorate,decoration={markings,
    mark=at position .5 with {\arrow[scale=1.3]{<}}}}] (6.2,0)
      +(0,-1) -- +(0,1) node[below,at
      start]{$j$};     \end{tikzpicture}& i=k=j+1\\
    -\begin{tikzpicture}[very thick,yscale=.6,xscale=.8,baseline]
     \draw[postaction={decorate,decoration={markings,
    mark=at position .5 with {\arrow[scale=1.3]{<}}}}] (6.2,0)
      +(1,-1) -- +(1,1) node[below,at start]{$k$}; \draw[postaction={decorate,decoration={markings,
    mark=at position .5 with {\arrow[scale=1.3]{<}}}}] (6.2,0)
      +(-1,-1) -- +(-1,1) node[below,at start]{$i$}; \draw[postaction={decorate,decoration={markings,
    mark=at position .5 with {\arrow[scale=1.3]{<}}}}] (6.2,0)
      +(0,-1) -- +(0,1) node[below,at
      start]{$j$};     \end{tikzpicture}& i=k=j-1\\
     0& \text{otherwise}
      \end{cases}
  \end{equation*}
\end{definition}
Let $e_\chi$ be the idempotent given by summing $e(\Bi)$ for all $\chi$-parabolic total orders satisfying \eqref{eq:order-1}.  For a non-total order satisfying \eqref{eq:order-1} and \eqref{eq:order-2}, we associate the ``divided power''  idempotent $e'(\Bi)$ which acts on each  equivalence class of consecutive strands by a primitive idempotent in the nilHecke algebra (for example, that introduced in \cite[2.18]{KLMS}).  So for example, we associate to $(3,2^{(2)},3,1,3)$ the idempotent
\[\begin{tikzpicture}[baseline,very thick,yscale=.5]
  \draw[wei] (-1, -1) to[out=90,in=-90] node[below,at start]{$3$}   (-1,1);
  \draw (-.5,-1) to[out=90,in=-90]   node[below,at start]{$2$} (0,1);
  \draw (0,-1) to[out=90,in=-90] node[below,at start]{$2$} node[pos=.8,circle,fill=black,inner sep=2pt]{} (-.5,1);
  \draw[wei] (.5,-1) to[out=90,in=-90] node[below,at start]{$3$}  (.5,1);
  \draw  (1,-1) to[out=90,in=-90] node[below,at start]{$1$}(1,1);
  \draw[wei] (1.5,-1) to[out=90,in=-90] node[below,at start]{$3$}  (1.5,1);
\end{tikzpicture}\]

Let $S_\chi$\notation{$S_\chi$}{The stabilizer of $\chi$ in $S_n$} be the stabilizer of $\chi$ in $S_n$; this naturally acts on the subalgebra $e_\chi \bT e_\chi$ by permuting groups of red strands (or equivalently, dots on those red strands).  
\begin{definition}
  Let \notation{$\bT^\chi$}{A deformation of the usual KLRW algebra.}$\bT^\chi=(e_\chi \bT e_{\chi})^{S_\chi}$ be the invariants of $S_\chi$ acting on the subalgebra $e_\chi \bT e_{\chi}$.
\end{definition} 
This is a canonical deformation of the algebra $\tilde{T}^{\bla}$ of \cite{Webmerged} attached to the sequence of dominant weights $\bla=(g_1\omega_{\rankm},\dots, g_k\omega_{\rankm})$ where $g_1,\dots, g_k$ are the sizes of the blocks of consecutive equal entries in $\chi$, i.e. $S_\chi=S_{g_1}\times \cdots\times  S_{g_k}$.  

Note that this algebra breaks up into a sum of subalgebras where we fix the number of strands with each label; as usual, we let $v_i$ denote the number with label $i$.  We'll be particularly interested in the case when $v_1\leq  v_2\leq \dots \leq v_m$. This is the condition that the corresponding weight of $\mathfrak{sl}_m$ is dominant; in our usual correspondence, it corresponds to the $n$-tuple $\nu=(1,\dots, 1,2,\dots, 2,\dots)$ where $i$ appears $v_i-v_{i-1}$ times.  
This same algebra is considered in \cite[\S 4]{KLSY} in the case $\rankp=2$ and is denoted $W(\mathbf{g},v_1)$ (using our $g_*$ and $v_*$ as above).

From its realization as a weighted KLR algebra, the algebra $\bT$ inherits a polynomial representation. 
\begin{definition}\label{def:poly-rep}
  The polynomial representation of $\bT$  is the vector space 
  \begin{equation*}
      \poly=\bigoplus_{\Bi}\K[Y_1, \dots, Y_{V}]e({\Bi}),
  \end{equation*}  with sum running over total orders on $\Omega$ satisfying \eqref{eq:order-1}.  The action is given by the rules:
\begin{itemize}
    \item $e({\Bi})$ acts by projection to the corresponding summand,
    \item a dot on the $k$th strand from the left acts by multiplication by $Y_k$,
    \item a crossing of the $k$th and $k+1$st strands with $\Bi$ at the bottom and $\Bi'$ at top acts by
    \begin{itemize}
        \item  If $i_k=i_{k+1}$, the divided difference operator \[fe_{\Bi}\mapsto \frac{f^{(k,k+1)}-f}{Y_{k+1}-Y_k}e_{\Bi'}.\]
        \item If $i_k+1=i_{k+1}$, the permutation $(k,k+1)$ followed by a multiplication \[fe_{\Bi}\mapsto (Y_{k+1}-Y_k) f^{(k,k+1)}e_{\Bi'}.\]    
        \item Otherwise, the permutation $(k,k+1)$  \[fe_{\Bi}\mapsto f^{(k,k+1)}e_{\Bi'}.\]
    \end{itemize}
\end{itemize}
The polynomial representation $\poly^\chi$ for $\bT^\chi$ is given by $(e_\chi \poly)^{S_\chi}$ where $S_\chi$ acts by permuting red dots as usual.  
\end{definition}

\subsection{Violating quotients}

\begin{definition}
  We call an idempotent $e(\Bi)$ {\bf violating} if $i_1\neq \rankp$; that is, if $(\rankp,1)$ is not minimal in $\prec$.
Let $\vT^\chi$\notation{$\vT^\chi$}{The quotient of $\bT^\chi$ by the 2-sided ideal generated by all violating idempotents.} be the quotient of $\bT^\chi$ by the 2-sided ideal generated by all violating idempotents.
\end{definition}

The algebra $\vT^\chi$ is not precisely the algebra $\bvT^\chi$ defined in \cite{Webmerged}, but a deformation of it which we've considered in several contexts, in particular, in \cite[\S 4]{Webunfurl}. This deformation is flat, since it is a special case of deforming the polynomials defining the KLRW algebra (as discussed in \cite[Prop. 2.23]{WebwKLR}); in the case $\rankp=2$, this is the redotted algebra discussed by Khovanov-Sussan in \cite{KSred}.  The algebra $\bvT^\chi$ is the quotient of $\vT^\chi$ by all red dots.
Since the red dots are central, and the polynomial ring is graded local, every gradable simple $\vT^\chi$-module factors through $\bvT^\chi$, and so the Grothendieck group of $\bvT^\chi\wgmod$ agrees with the Grothendieck group of  $\vT^\chi\wgmod$.  From \cite[\S 4]{Webunfurl}, we have:
\begin{theorem}
	The categories of $\vT^\chi$-modules and $\bvT^\chi$-modules are categorifications of $\Sym^{\mathbf{g}}(\C^m)$, with the categorical $\mathfrak{sl}_m$-action given by induction and restriction functors changing the number of black strands.  
\end{theorem}

\subsection{Ladder bimodules}
In our notation, we identify the dominant weight $\chi$ with a weight of  $\mathfrak{sl}_\infty$ by $\mu_{\chi}=\sum_{i=1}^n \epsilon_{\chi_i}$.  This is an injective map, but is far from surjective, since it only hits weights where the coefficients of the $\ep_i$'s are positive and sum to $n$ (in the usual parlance, they are level $n$).  In particular, \[\al_j^\vee(\mu_{\chi})=\#\{j\mid \chi_j=i+1\}-\#\{j \mid \chi_j=i\}.\] 
 \begin{definition}
 \notation{$\chi^{\pm i}$}{The sequence $\chi$ with an entry $i$ changed to $i\pm 1$ if such a dominant weight exists.}Let $\chi^{+ i}$ denote $\chi$ with an entry $i$ increased to $i+1$ if such a dominant weight exists, and $\chi^{-i}$ denote $\chi$ with an entry $i+1$ decreased to $i$ if such a dominant weight exists. Let $\chi^{\pm i^a}$ be the result of doing this operation $a$ times.
\end{definition} 
These operations are uniquely characterized by the fact that:
\begin{lemma}
If $\mu_{\chi^{\pm i}}$ and $\mu_{\chi}$ exist, then $\mu_{\chi^{\pm i}}=\mu_{\chi}\pm\al_i$.  
\end{lemma}
\begin{proof}
If $\chi^{+ i}$ exists, then for some $k$, we have  $i=\chi_k<\chi_{k+1}$.  In this case, the dominant weight \[\chi^{+ i}=\chi+\ep_k=(\chi_1,\dots, \chi_k+1,\dots,\chi_n)\] satisfies
\[\al_j^\vee(\chi+\ep_k)=\begin{cases} \al_j^\vee(\mu_{\chi})-1 & j=\chi_k\pm 1,\\
\al_j^\vee(\mu_{\chi})+2 & j=\chi_k,\\
\al_j^\vee(\mu_{\chi}) & \text{otherwise}
\end{cases}
\]
Thus, we have that $\mu_{\chi+\ep_k}=\mu_\chi+\al_{i}$ as desired.  The second half of the result follows from the fact that $(\chi^{+ i})^{-i}=\chi$.
\end{proof}

Assume that $i$ appears at least $a$ times in $\chi$.  Let $\chi'=\chi^{+i^a}$ be the dominant weight obtained by changing $a$ instances of $i$ to $i+1$.  Let $S_{\chi,\chi'}=S_{\chi}\cap S_{\chi'}$.  
\begin{definition}
  The {\bf ladder bimodules} are the subspace \notation{$\mathbb{E}_i^{(a)}, \mathbb{F}_i^{(a)}$}{The ladder bimodules defining a categorical action.}$\mathbb{E}_i^{(a)}=(e_\chi \bT e_{\chi'})^{S_{\chi,\chi'}}$ considered as a $\bT^{\chi'}$-$\bT^\chi$ bimodule, and similarly $\mathbb{F}_i^{(a)}=(e_{\chi'} \bT e_\chi)^{S_{\chi,\chi'}}$ simply swaps the roles of $\chi$ and $\chi'$.
\end{definition}
These are generalizations of the ladder bimodules defined in \cite[\S 5.2]{KLSY}.
We draw these by pinching together the red strands in a single equivalence class under the $\chi$-parabolic preorder giving an idempotent into a single red strand.  Thus, elements of this bimodule look like:
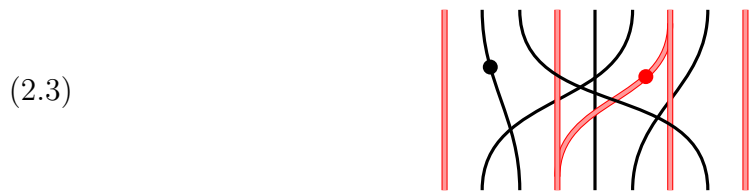
\begin{equation}
    \tikz[baseline, yscale=1.2]{\draw[wei] (-2,-1) to[out=90,in=-90] (-2,1);\draw[very thick] (-1.5,-1)to[out=90,in=-90](.5,1);\draw[very thick] (-1,-1)to[out=90,in=-90] node [pos=.7,circle,fill=black,inner sep=2pt] {} (-1.5,1);\draw[very thick] (0,-1)--(0,1);\draw[very thick] (.5,-1)to[out=90,in=-90](1.5,1);
    \draw[wei] (-.5,-.85) to[out=90,in=-90] node [pos=.7,circle,fill=red,inner sep=2pt] {} (1,.85) ;\draw[wei] (1,-1)--(1,1); \draw[wei] (-.5,-1)--(-.5,1);\draw[very thick] (1.5,-1)to[out=90,in=-90](-1,1);\draw[wei] (2,-1)--(2,1); }\label{eq:ladder}
\end{equation}

These bimodules have a ``representation'' as well.  By a representation of a bimodule $B$ over algebras $A_1$ and $A_2$, we mean a representation of the Morita context:
\[A_B=\begin{bmatrix} A_1 & B \\ 0& A_2\end{bmatrix}\]
with the obvious matrix multiplication.  That is, a left module $V_i$ of $A_i$, and a bimodule map $B\to \Hom_{\K}(V_2, V_1)$.  In our case, we will use the polynomial representations $\poly^\chi$ of $\bT^\chi$.  Diagrams other than the split and join of red strands act by the formulas in Definition \ref{def:poly-rep}.  The formulas for splitting and joining are the same as in  \cite{SWschur}:
\begin{itemize}
    \item split corresponds to the inclusion $\poly^\chi \hookrightarrow \poly^{\chi,\chi'}$, where the latter is the invariants under $S_{\chi,\chi'}\subset S_{\chi}$.
    \item merge corresponds to the divided difference operator $\poly^{\chi,\chi'}\to \poly^{\chi'}$ given by 
    \[f\mapsto \sum_{\sigma\in S_{\chi'}/S_{\chi,\chi'}}\frac{\displaystyle f^{\sigma}}{\Delta_{+i^{(a)}}^\sigma},\]
    where $\Delta_{+i^{(a)}}$ is the product of $Y_k-Y_\ell$ where $k$ ranges over the red strands with $\la_{n,j}=i+1$, and $\ell$ over the red strands in the ``rung'' of the ladder.  This corresponds to the operator in equivariant cohomology which integrates a $P_{\chi,\chi'}$-equivariant class over $P_{\chi'}/P_{\chi,\chi'}$ to give a $P_{\chi'}$-equivariant class.  In terms of the nilHecke algebra, this corresponds to the diagram:
    
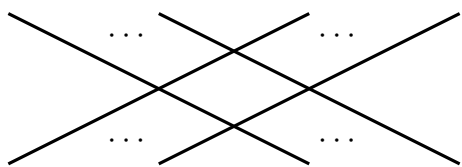
\[\tikz[very thick]{\draw (-3,-1) -- (1,1);\draw (-1,-1) -- (3,1);\draw (3,-1) -- (-1,1);\draw (1,-1) -- (-3,1); \node at (1.4,.7){$\cdots$};\node at (-1.4,.7){$\cdots$};\node at (1.4,-.7){$\cdots$};\node at (-1.4,-.7){$\cdots$};}\]
\end{itemize}
\begin{lemma}\label{lem:E-rep}
The formulas above define a representation of the bimodule $\mathbb{E}_i^{(a)}$.
\end{lemma}
\begin{proof}
Here, we use the fact that $\mathbb{E}_i$ by definition is the subbimodule of $e_{\chi'}\bT e_{\chi}$ invariant under $S_{\chi,\chi'}$.  This embedding corresponds to taking a diagram as in \eqref{eq:ladder}, and simply expanding red strands.
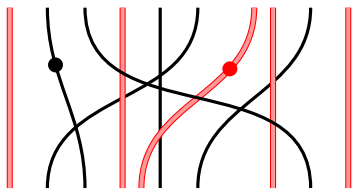
\begin{equation*}
    \tikz[baseline, yscale=1.2]{\draw[wei] (-2,-1) to[out=90,in=-90] (-2,1);\draw[very thick] (-1.5,-1)to[out=90,in=-90](.5,1);\draw[very thick] (-1,-1)to[out=90,in=-90] node [pos=.7,circle,fill=black,inner sep=2pt] {} (-1.5,1);\draw[very thick] (0,-1)--(0,1);\draw[very thick] (.5,-1)to[out=90,in=-90](2,1);
    \draw[wei] (-.25,-1) to[out=90,in=-90] node [pos=.7,circle,fill=red,inner sep=2pt] {} (1.25,1) ;\draw[wei] (1.5,-1)--(1.5,1); \draw[wei] (-.5,-1)--(-.5,1);\draw[very thick] (2,-1)to[out=90,in=-90](-1,1);\draw[wei] (2.5,-1)--(2.5,1); }
\end{equation*}
This does not precisely match the operators above, but it does after we add a crossing of the red strands that joined at the top:
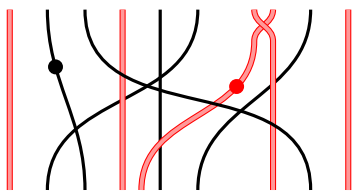
\begin{equation*}
    \tikz[baseline, yscale=1.2]{\draw[wei] (-2,-1) to[out=90,in=-90] (-2,1);\draw[very thick] (-1.5,-1)to[out=90,in=-90](.5,1);\draw[very thick] (-1,-1)to[out=90,in=-90] node [pos=.7,circle,fill=black,inner sep=2pt] {} (-1.5,1);\draw[very thick] (0,-1)--(0,1);\draw[very thick] (.5,-1)to[out=90,in=-90](2,1);
    \draw[wei] (-.25,-1) to[out=90,in=-90] node [pos=.75,circle,fill=red,inner sep=2pt] {} (1.25,.65) to[out=90,in=-90] (1.5,1) ;\draw[wei] (1.5,-1)to[out=90,in=-90](1.5,.65) to[out=90,in=-90](1.25,1); \draw[wei] (-.5,-1)--(-.5,1);\draw[very thick] (2,-1)to[out=90,in=-90](-1,1);\draw[wei] (2.5,-1)--(2.5,1); }
\end{equation*}
The action of this in the usual polynomial representation of the KLR algebra of $A_{\rankp}$ matches the formulas we have given.  The fact that this is a bimodule map follows from the faithfulness of the polynomial representations.
\end{proof} \section{The geometry of quivers and perverse sheaves}

\subsection{Quiver representations} As in the previous section, we fix  $\rankp,n$ and $\chi$.  We can think of this as giving a cocharacter into $GL_n$, and let \notation{$P_\chi$}{The parabolic associated to $\chi$; this is block upper-triangular matrices with blocks corresponding to repeats of $\chi$.}$P_\chi\subset GL_n$ be the parabolic whose Lie algebra is the non-positive weight space for this cocharacter.  This is the standard Borel if $\chi_i\neq \chi_j$ for all $i,j$, and in general is block upper-triangular matrices, with blocks corresponding to the consecutive $\chi_i$ which are equal. 

Fix a dimension vector $\Bv=(v_1,\dots, v_{\rankp-1})\in \Z_{\geq 0}^{\rankp-1}$ as in the previous section.  Consider the representation 
\[V=\Hom(\C^{v_1},\C^{v_2})\oplus \Hom(\C^{v_2},\C^{v_3})\oplus \cdots \oplus \Hom(\C^{v_{\rankp-2}},\C^{v_{\rankp-1}})\oplus \Hom(\C^{v_{\rankp-1}},\C^n)\]  of the group \notation{$\Gc{\chi}$}{The group $P_\chi\times
\Gz$ reflecting change of basis for quiver representations together with a flag.}$G=\Gc{\chi}=P_\chi\times GL_{v_1}\times \cdots \times GL_{v_{\rankm}}$.  Recall that we introduced the notation \notation{$\Gz$}{The group $GL_{v_1}\times \cdots \times GL_{v_{\rankm}}$.}$\Gz=GL_{v_1}\times \cdots \times GL_{v_{\rankm}}$ in the introduction.
In the notation popular with physicists, this corresponds to the following quiver: 
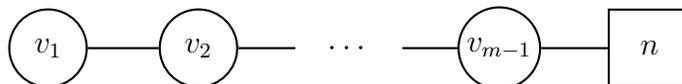
\begin{equation*}
\tikz{
	\node[draw, thick, circle, inner sep=6pt,fill=white] (a) at (-6,0) {$v_1$};
	\node[draw, thick, circle, inner sep=6pt,fill=white] (b) at (-4,0) {$v_2$};	
		\node[draw, thick, circle, inner sep=2pt,fill=white] (d) at (0,0) {$v_{\rankm}$};
	\node[draw, thick,inner sep=12pt] (s) at (2,0){$n$};
	\node[inner sep=12pt,fill=white] (c) at (-2,0){$\cdots$};
		\draw[thick,mid] (a) -- (b) ;
		\draw[thick,mid] (b) -- (c) ;
		\draw[thick,mid] (c) -- (d) ;
		\draw[thick,mid] (d) -- (s) ;
}
\end{equation*}
Attached to the space $V$ with the action of $G$, we have an equivariant derived category $D^b_G(V)$\notation{$D^b_G(V)$}{The derived category of constructible sheaves on the quotient stack $V/G$---that is, the equivariant derived category for $G$ acting on $V$.} as introduced in \cite{BL}; in modern terminology, we would think of this as the derived category of constructible sheaves on the quotient stack $V/G$.  There are various other avatars of this category, such as strongly equivariant D-modules, but we will only use a few basic facts about this category, such as the decomposition theorem and the computation of Ext-algebras as Borel-Moore homology familiar from \cite[\S 8.6]{CG97}.

First, we simply need to classify the orbits of $P_\chi$ in $V$.  Recall that a {\bf segment} in $[1,\rankp]$ is a list of consecutive integers $(k,k+1,\dots, \ell)$, and a multi-segment is a multi-set of segments.  The {\bf dimension vector} of a segment is the vector $(0,\dots, 1,\dots, 1,\dots, 0)\in \Z^{\rankp}$ with 1 in every position in $[k,\ell]$ and 0 in all others, and the dimension vector of a multi-segment is the sum of those for the constituent segments.  That is, it is the vector that records how many times an index $i\in [1,\rankp]$ appears in the constituent segments.
\begin{definition}\label{def:segments}
A flavored segment is a pair consisting of a segment and an integer $\beta\in \Z$.  A $\chi$-flavored multi-segment is a multi-segment with a choice of flavoring on each segment with $\ell=\rankp$ (and {\it no} additional information about other segments) such that the flavors of the different segments agree with $\chi$ up to permutation.

We call the segments with $\ell=\rankp$ {\bf flavored} and those with $\ell<\rankp$ {\bf unflavored}.
\end{definition}
\begin{example}
If $n=2$ and $\rankp=2$ and $v_1=1$ then there are two multisegments with the correct dimension vector: $\{(1), (2),(2)\}$ and $\{(1,2), (2)\}$.  There is only one way of flavoring $\{(1), (2),(2)\}$, mapping the two copies of $(2)$ to the two coordinates of $\chi$. 

On the other hand, for $\{(1,2), (2)\}$, there are two different possible flavors, as long as $\chi_1\neq \chi_2$,  depending on the bijection we choose between the sets $\{(1,2), (2)\}$ and $\{\chi_1,\chi_2\}$.  If $\chi_1=\chi_2$, then we are back to having a single possible choice of flavor.
\end{example}
These are relevant because of the following fact from the appendix:
\begin{theorem}[Lemmata \ref{lem:orbits} \& \ref{lem:simply-connected}]
The $P_\chi$-orbits in $V$ are in bijection with $\chi$-flavored multi-segments with the corresponding dimension vector.  Each of these orbits is equivariantly simply connected.  
\end{theorem}

We view the subset $\{(i,1),\cdots,(i,v_i)\}\subset \Omega $ as corresponding to an ordered basis 
$\{b_{(i,1)},\dots, b_{(i,v_i)}\}$ of $\C^{v_i}$.  Giving this copy of $\C^{v_i}$ degree $i$, we can view
\[\C^\Omega \cong \C^{v_1}\oplus \cdots \oplus \C^{v_{\rankp}}\] as a graded vector space, and we can view a degree 1 map $f\colon \C^\Omega \to \C^\Omega $ as an element of $V$, that is, of quiver representation of $A_{\rankp}$ with dimension vector $(v_1,\dots, v_{\rankp})$.  

\subsection{Quiver flag varieties and the equivariant derived category} As before, consider a total preorder $\preceq$ satisfying \eqref{eq:order-1} and  \eqref{eq:order-2}. This choice of preorder induces a flag $F^{\prec}_\bullet$, with each subspace given by the formula:
\[F^{\prec}_{(i,k)}=\operatorname{span}\{b_{(j,\ell)}\mid (j,\ell)\preceq (i,k)\}.\]
If the preorder is not an order,  equivalent elements give the same subspace, so this flag will have some redundancies in it.
We say that a flag on $\C^\Omega $ indexed by the equivalence classes   has {\bf type $\preceq$} or {\bf type $\Bi$} if it is conjugate to a flag of this form under $\Gz$, and let  $\Fl(\Bi)$ be the set of such flags.  Note that $\Gz$ acts on this space transitively, with the stabilizer of $F^{\prec}_\bullet$ given by a parabolic $P_0$, which only depends on the equivalence relation of $\preceq$; in particular, for any total order, we get the same Borel $B_0$.  Consider the $\Gz$-space  
\[X(\Bi)=\{ (f,F_\bullet)\in V\times \Fl(\Bi) \mid f(F^{\prec}_{(i,k)})\subset F^{\prec}_{(i,k)}\}.\notation{$X(\Bi)$}{The set of quiver representations with a flag of type $\Bi$.}\]
\begin{lemma}
If $\Bi$ is $\chi$-parabolic, then $X(\Bi)$ has an action of $P_\chi$ by the post-composition action on $V$ and the trivial action on $\Fl(\Bi)$.  This commutes with the $\Gz$-action, inducing a $P_\chi\times \Gz$ action for which projection to $V$ is equivariant.
\end{lemma}
\begin{proof}
Assume that $\Bi$ is $\chi$-parabolic.  Thus, we have that if $\chi_k=\chi_{k+1}$, then  $F^{\prec}_{(\rankp,k+1)}=F^{\prec}_{(\rankp,k)}+\C\cdot b_{\rankp,k+1}$.  By degree considerations, $f(b_{\rankp,k+1})=0$, so $f(F^{\prec}_{(\rankp,k)})\subset F^{\prec}_{(\rankp,k)}$, then automatically, we have
\[f(F^{\prec}_{(\rankp,k+1)})=f(F^{\prec}_{(\rankp,k)})\subset F^{\prec}_{(\rankp,k)}\subset F^{\prec}_{(\rankp,k+1)}.\]
Thus, for a $\chi$-parabolic flag, we only need to check that $f(F^{\prec}_{(i,k)})\subset F^{\prec}_{(i,k)}$ for $i<\rankp$ or when $i=\rankp$ and $\chi_k\neq \chi_{k-1}$.  

Consider  $(f,F_\bullet)\in X(\Bi)$ and $g\in P_\chi$.  Consider the map $gf\colon \C^{\Omega} \to \C^{\Omega}$.  This is again a quiver representation, which is compatible with the flag $gF^{\prec}_{(i,k)}$.  Since $g\in P_\chi$, we have that $gF^{\prec}_{(i,k)}=F^{\prec}_{(i,k)}$ if $i<\rankp$ or if $i=\rankp$ and $\chi_k\neq \chi_{k-1}$.  Thus, by our observation above, our original flag is still compatible with $gf$.  The fact that this commutes with $\Gz$ is clear.
\end{proof}

For each segment $(k,\dots, \ell)$, 
let $\Bi_{(k,\dots, \ell)}=(\ell, \ell-1,\dots, k)$ be the word where we list the entries in reverse order.  
\begin{definition}
For a $\chi$-flavored multi-segment $\mathbf{Q}$, the corresponding {\bf good word} $\Bi_{\bQ}$\notation{$\Bi_{\bQ}$}{The good word corresponding to a $\chi$-flavored multi-segment $\mathbf{Q}$.} is the result of concatenating \begin{enumerate}
    \item the words for the unflavored segments in increasing lexicographic order (with the convention that attaching any suffix makes a word lower).  
    \item the words for the flavored segments sorted first by the attached flavor (in increasing order), and with the entries of the words for a single flavor shuffled together in decreasing order.  
\end{enumerate} 
\end{definition}
\begin{example}
Let $n=2,\rankp=3$ and consider the multi-segment \[\bQ=\{(1),(2),(2,1), (3,2,1) ,(3,2)\}.\]  If $\chi_1>\chi_2$, a flavoring of this multi-segment is a bijection between the sets $\{(3,2,1) ,(3,2)\}$ and $\{\chi_1,\chi_2\}$. In this case, good words for this multi-segment with the flavorings for the order we've written the sets above and its opposite are, respectively:
\[(1,2,1,2,3,2,1,3,2)\qquad (1,2,1,2,3,2,3,2,1).\]
On the other hand, if $\chi_1=\chi_2$, there is only one possible flavoring with
\[(1,2,1,2,3,3,2,2,1).\]
\end{example}
\begin{lemma}
The good word $\Bi_{\bQ}$ is always $\chi$-parabolic, and the image of $X(\Bi_{\bQ})$ is precisely the closure of the corresponding $\Gc{\chi}$-orbit.
\end{lemma}
\begin{proof}
We prove this by induction on the number of segments.  Note that since $X(\Bi)$ is irreducible, the same is true of its image in $V$, so its image is the closure of some orbit.  

Consider the segment $(k,\dots, \ell)$ which appears first in the good word; if any unflavored segments appear in $\bQ$, then this will be unflavored.  Assume for now that $\ell< \rankp$.  This portion of the word gives a submodule $M\subset \C^\Omega$, and on an open subset of $X(\Bi)$, this submodule is the unique indecomposible module with this dimension vector. Also, by assumption, the quotient $\C^\Omega/M$ gives a point in $X(\Bi')$, the good word obtained by removing this segment. By induction, on an open subset of $X(\Bi)$, the quotient $\C^\Omega/M$ has the representation type given by $\bQ$ with this segment removed.  

By the lexicographic condition $\Ext^1(\C^\Omega/M,M)=0$, so on the open set where both $M$ and $\C^\Omega/M$ have the correct representation type, we have a split extension, and the result follows.  

We need to be a bit careful in the case of a flavored word with $\chi_1=\chi_{2}=\cdots =\chi_{p}$; in this case, we don't have a single segment appearing at the bottom, but rather a word of the form 
$(\rankp,\dots, \rankp,\rankp-1,\dots, \rankp-1,\dots)$, which again has a corresponding subrepresentation $M$, which by construction satisfies $M\cap \C^n=\operatorname{span}(b_{\rankp,1},\dots, b_{\rankp,p})$.  The generic representation with this dimension vector is lies in the orbit given by the corresponding multi-segment, and by induction, the same is true $\C^\Omega/M$.  In particular, $\C^{\Omega}/M$ is compatible as desired with the action of $P_{\chi'}$ where $\chi'=(\chi_{p+1},\dots, \chi_n)$.  As before, we have $\Ext^1(\C^\Omega/M,M)=0$, so generically on $X(\Bi)$, we have a split extension, and the desired generic representation type.
\end{proof}

Let $\pi^{\Bi}\colon X(\Bi)\to V$ be the projection map, and $\perv_{\Bi}=\pi_*^{\Bi}\C_{X(\Bi)}[\dim X(\Bi)]$ be the pushforward of the sheaf of locally constant  $\C$-valued functions on $X(\Bi)$.\notation{$\perv_{\Bi}$}{The pushforward of the sheaf of locally constant  $\C$-valued functions on $X(\Bi)$ for the projection map $\pi^{\Bi}\colon X(\Bi)\to V$.} 
\begin{lemma}
If $\Bi'$ is the totalization of $\Bi$,  then there is a natural map $\phi\colon X(\Bi')\to X(\Bi)$ satisfying $\pi^{\Bi}\circ \phi =\pi^{\Bi'}$.  The map $\phi$ is a fiber bundle with fiber given by a product of complete flag varieties.
Thus, $\perv_{\Bi'}$ is a sum of copies of shifts of $\perv_{\Bi}$.   
\end{lemma}
\begin{proof}
The map $\phi$ is defined by forgetting the spaces attached to elements which are not maximal in their equivalence class in $\Bi$.  This  indeed gives an element of $X(\Bi)$ and a point in the fiber over a given point in $X(\Bi)$ is given by choosing a total flag in the subquotient of consecutive spaces in the flag.  This has an induced grading, for which it is homogeneous of a single degree.  Thus, the induced action of $f$ is trivial for degree reasons, and any choice of total flag gives an element of $X(\Bi').$

Since the map $\phi$ is a fiber bundle whose fiber is a smooth projective variety, the pushforward $\phi_*\C_{X(\Bi')}[\dim X(\Bi')]$ is a sum of shifts of local systems by the Hodge theorem.  The space $X(\Bi')$ is homotopy equivalent to $\Gz/P_0$ which is simply connected, so these local systems are all trivial.  This gives the result.   
\end{proof}

Recall that for each $\Gc{\chi}$-orbit, there is a unique $\Gc{\chi}$-equivariant perverse sheaf $\mathbf{IC}_{\bQ}$ which extends the trivial local system on the orbit.  

\begin{theorem}\label{thm:all-ICs}
Every simple $\Gc{\chi}$-equivariant perverse sheaf on $V$ is the intersection cohomology complex $\mathbf{IC}_{\bQ}$ for the orbit of a flavored multi-segment with the trivial local system, and $\mathbf{IC}_{\bQ}$ is a summand of $\perv_{\Bi_{\bQ}}$ up to shift.\notation{$\mathbf{IC}_{\bQ}$}{The intersection cohomology complex for the orbit of a flavored multi-segment with the trivial local system.}
\end{theorem}
\begin{proof}
The fact that $\mathbf{IC}_{\bQ}$ is a complete list of simple perverse sheaves follows from the classification of orbits and their equivariant 1-connectedness.  

Since $\pi^{\Bi}$ is a proper $\Gc{\chi}$-equivariant map with $X(\Bi)$ smooth, the Decomposition Theorem shows that $\perv_{\Bi}$ is a Verdier self-dual direct sum of shifts of perverse sheaves supported on the orbit closure $\overline{\mathbb{O}_{\bQ}}$.  

At least one of these summands must have support precisely equal to $\overline{\mathbb{O}_{\bQ}}$, and $\mathbf{IC}_{\bQ}$ is the only such option.  Thus, it must appear as a summand up to shift. 
\end{proof} 

\subsection{Graded lifts}
\nc{\Wpre}{W^{\operatorname{pre}}}
Let $\perv_{\chi}=\oplus_{\Bi}\perv_{\Bi}$ where $\Bi$ ranges over totalizations of $\chi$-parabolic words and $X_{\chi}=\sqcup_{\Bi} X(\Bi)$.  
As usual, by \cite{CG97}, we have that 
\[A_\chi=H^{BM,\Gc{\chi}}_*(X_{\chi}\times_V X_{\chi})\cong \Ext^\bullet_{D^b_{\Gc{\chi}}(V)}(\perv_{\chi},\perv_{\chi})\notation{$A_\chi$}{The Ext algebra $\Ext^\bullet_{D^b_{\Gc{\chi}}(V)}(\perv_{\chi},\perv_{\chi})$.}\] as  dg-algebras, where we give the left-hand side a non-standard grading. 

The sheaves $\perv_{\chi}$ have an additional structure which is hard to see directly on the sheaf level: they are the solutions of complexes of D-modules $\mathsf{F}_{\chi}$ which carry a mixed Hodge structure.  See \cite{saito2016young} for more details on mixed Hodge structures.   We don't need much of the full power of this mixed Hodge structure, but as in \cite[\S 2.8]{WebSD}, this gives us a natural dgg (differential-graded and graded) structure on $A_{\chi}$.  A dgg-algebra $A$ is a dg-algebra in the category of graded vector spaces;  that is, $A$ has a homological grading and an internal grading, with the differential having degree $(1,0)$. Given a dgg-algebra $A$, we let $A\dggmod$ be the category of representations of $A$, again, in the category of graded vector spaces, that is, dg-modules over $A$ equipped with an additional grading.  These form a graded dg-category where the morphism complexes are the degree $(*,0)$ part of the usual morphism complexes of dg-modules---again, this just means that we only use degree 0 morphisms between graded vector spaces.

One useful observation is that a dgg-algebra where the internal grading coincides with the homological grading is necessarily formal, since it is quasi-isomorphic to a graded $A_{\infty}$ algebra, and all the higher operations are trivial for grading reasons (see \cite[Th. 2.38]{WebSD} and \cite[Prop. 5.4]{ciriciFilteredAinfinity2022}).  

We can use this dgg-algebra structure to construct a graded lift:
\begin{definition}
	A {\bf grading} on an object $M$ in 
$D^{b}_{\Gc{\chi}}(V)$ is a dgg-module $N$ over $A_{\chi}$ and a quasi-isomorphism of $A_{\chi}$ dg-modules $\Hom^{\bullet}(\perv_{\chi},M)\cong N$, forgetting the second grading on $N$.  We can define a graded lift of 
$D^{b,\operatorname{mix}}_{\Gc{\chi}}(V)$ whose:
\begin{itemize}
	\item objects are graded objects in $D^{b}_{\Gc{\chi}}(V)$
	\item morphisms are degree 0 morphisms in $D^{b,\operatorname{mix}}_{\Gc{\chi}}(V)$ under the isomorphism \[\Hom^{\bullet}_{D^{b}_{\Gc{\chi}}(V)}(M,M')\cong \Hom^{\bullet}_{A\dgmod}(N,N').\]  That is, the functor sending a graded sheaf to the corresponding dgg-module is fully faithful.
\end{itemize} 
\end{definition}
\begin{remark}
	It should be more geometrically natural to define a grading as a Hodge, or perhaps twistor, structure on the sheaf $M$.  Unfortunately, the existence of non-Tate Hodge/twistor structures means that there are ``too many'' such structures lifting a given dgg-module structure.  It seems likely that a canonical ``most Tate'' choice of Hodge structure exists, but constructing it is too complicated a job for this paper.  This comes down to the question of whether the Hodge structure on $A_{\chi}$ makes this into a formal dg-algebra in the category of mixed Hodge structures.  See \cite[Rmk.\ 2.59]{WebSD} for more discussion of this phenomenon.
\end{remark}

By \cite[Thm. 4.5]{WebwKLR}, we have that:
\begin{theorem}\label{thm:T-iso}
For any preorders $\Bi,\Bj$ satisfying \eqref{eq:order-1} and \eqref{eq:order-2}, we have an isomorphism 
\[\Ext^\bullet(\perv_{\Bi},\perv_{\Bj})\cong e'(\Bi)\bT^\chi e'(\Bj)\]
matching convolution product with multiplication in $\bT^\chi$, the red dots with the cohomology ring $H^*(BP_\chi)$ and the black dots with Chern classes of the tautological bundles on $X(\Bi)$.  In particular, we have an isomorphism $A_\chi\cong \bT^\chi$.  
\end{theorem}
It might seem strange that in the definition of the algebra $A_\chi$, we used only honest orders, not preorders, when the first sentence of the theorem covers preorders. This is needed to match the usual definition of $\bT^\chi$.  If we incorporated preorders, the result would be a larger, Morita equivalent algebra which contained thick strands as in the extended graphical calculus of \cite{KLMS,SWschur}.  It is much more difficult to give a nice presentation of this larger algebra.

These results combine to show that the algebra $\bT^\chi$ controls the derived category $D^b_{\Gc{\chi}}(V)$:
\begin{theorem}\label{thm:generator}
The object $\perv_{\chi}$ is a compact generator of $D^b_{\Gc{\chi}}(V)$.  Thus, the functor $\Ext(\perv_{\chi},-)$ induces equivalences:
\[D^b_{\Gc{\chi}}(V)\cong \bT^\chi\dgmod\qquad D^{b,\operatorname{mix}}_{\Gc{\chi}}(V)\cong\bT^\chi\dggmod\cong  D^b(\bT^\chi\operatorname{-gmod}).\]
The functor of forgetting graded structure corresponds to the functor considering a complex of graded modules as a dg-module by collapsing gradings.
\end{theorem}

\subsection{The categorical action}

Consider the 2-category $\mathsf{Perv}$\notation{$\mathsf{Perv}$}{The 2-category whose 1-morphisms are semi-simple $P_{\chi'}\times P_{\chi}$-equivariant perverse sheaves on $GL_n$.} whose
\begin{itemize}
    \item objects are dominant integral weights $\chi=(\chi_1\leq \cdots \leq \chi_n)$,
    \item 1-morphisms $\chi\to \chi'$ are sums of shifts of semi-simple $P_{\chi'}\times P_{\chi}$-equivariant perverse sheaves on $GL_n$, or equivalently, $GL_n$-equivariant perverse sheaves on $GL_n/P_{\chi'}\times GL_n/P_{\chi}$ with composition given by convolution,
    \item 2-morphisms in the equivariant derived category. 
\end{itemize}
As discussed in \cite[Thm. 6]{Webcomparison}, this is equivalent to the category  $\mathsf{Flag}$\notation{$\mathsf{Flag}$}{The 2-category whose 1-morphisms are singular Soergel bimodules over $H^*(BP_{\chi'})\times H^*(BP_{\chi})$} whose
\begin{itemize}
    \item objects are dominant integral weights $\chi=(\chi_1\leq \cdots \leq \chi_n)$,
    \item 1-morphisms $\chi\to \chi'$ are sums of singular Soergel  $H^*(BP_{\chi'})\operatorname{-}H^*(BP_{\chi})$-bimodules.
    \item 2-morphisms are degree 0 homomorphisms of bimodules. 
\end{itemize}
via the functor that takes hypercohomology of a sum of shifts of equivariant perverse sheaves.  For technical reasons, we add to these categories an object $\emptyset$, which has Hom with any other object given by the zero category.  

Consider the fiber product \[P_{\chi'}\setminus GL_n/P_{\chi}\times_{BP_{\chi}}V/\Gc{\chi}=\frac{GL_{n}\times V}{P_{\chi'}\times \Gc{\chi}}.\] where the action is via $(p',p)\cdot (g,v)=(p'gp^{-1},pv)$.  This is, of course, equipped with an action map \[a\colon P_{\chi'}\setminus GL_n/P_{\chi}\times_{BP_{\chi}}V/\Gc{\chi}\to V/\Gc{\chi'}, \] and as usual, we define convolution of sheaves $\mathcal{F}\in D^b_{P_{\chi'}\times P_\chi}(GL_n)$ and $\mathcal{G}\in D^b_{\Gc{\chi}}(V)$ as $\mathcal{F}\star \mathcal{G}=a_* (\mathcal{F}\boxtimes \mathcal{G}).$
\begin{lemma}
Convolution of sheaves induces a representation of $\mathsf{Perv}$ in $\mathsf{Cat}$ sending $\chi\mapsto D^b_{\Gc{\chi}}(V)$.
\end{lemma}

The category  $\mathsf{Flag}$ carries a well-known categorical action defined by Khovanov and Lauda \cite[\S 6]{KLIII}. 
We define $\mathcal{U}$\notation{$\mathcal{U}$}{The 2-category categorifying $U_q(\mathfrak{sl}_\infty)$} to be the 2-category categorifying the quantum group of $\mathfrak{sl}_\infty$ (following the definition of \cite[Def. 1.1]{Brundandef} which unifies all previous definitions).  In this category: 
\begin{itemize}
    \item objects are integral weights of $\mathfrak{sl}_\infty$, that is, finite sums of the unit vectors $\ep_k$ for $k\in \Z$.
    \item 1-morphisms are generated by formal symbols
    \[\eE_i\colon \mu \to \mu+\al_i=\mu+\ep_{i+1}-\ep_{i}\qquad \eF_i\colon \mu \to \mu-\al_i=\mu-\ep_{i+1}+\ep_{i}\]
    \item 2-morphisms are given by certain string diagrams modulo the relations in \cite{Brundandef}. 
\end{itemize}

As discussed in \cite[\S 2.2]{Webcomparison}, we thus have a 2-functor $\Phi_{\mathsf{P}}\colon \tU\to \mathsf{Perv}$\notation{$\Phi_{\mathsf{P}}$}{The 2-functor mapping $\tU$ to $\mathsf{Perv}$.} sending $\mu_\chi \to \chi$ for weights of the correct form, and to $\emptyset$ otherwise. 

If $\chi'=\chi^{+i^a}$, then the $P_{\chi}\times P_{\chi'}$-orbits on $G$ are in bijection with orbits of double cosets for $(S_\chi',S_\chi)$.  In particular, there is a unique closed orbit, given by the product $P_{\chi',\chi}=P_{\chi'}P_{\chi}$.  The cohomology $H^*_{P_{\chi'}\times P_{\chi}}(P_{\chi',\chi})$ of this orbit gives the action of \cite[\S 6]{KLIII}.  As discussed below the proof of Theorem 6 in \cite[\S 2.2]{Webcomparison}, we have that:  
\begin{lemma}
The image $\Phi_{\mathsf{P}}(\eE_i^{(a)})$ is the Verdier-self-dual shift of the constant sheaf $\C_{\chi',\chi}$ on $P_{\chi',\chi}$.  Similarly $\Phi_{\mathsf{P}}(\eF_i^{(a)})$ is the same sheaf with the role of the factors switched.
\end{lemma}

For us, the important new ingredient here is that since we have an algebraic manifestation of the category $D^b_{P_\chi}(V)$ given by the algebra $\bT^\chi$, the action of the 2-category $\tU$ has a similar manifestation.

For every 1-morphism $\mathcal{F}\colon \chi\to \chi'$ in the 2-category $\Perv$, we can define a bimodule 
\[B_{\mathcal{F}}=\Ext_{D^b_{\Gc{\chi'}}(V)}(\perv_{\chi'}, \mathcal{F}\star \perv_{\chi}).\notation{$B_{\mathcal{F}}$}{The bimodule corresponding to a 1-morphism in the 2-category $\Perv$.} \]
By Theorem \ref{thm:generator}, we have a commutative diagram   \[\tikz[->,thick]{
\matrix[row sep=18mm,column sep=35mm,ampersand replacement=\&]{
\node (d) {$D^{b,\operatorname{mix}}_{\Gc{\chi}}(V)$}; \& \node (e)
{$ D^{b,\operatorname{mix}}_{\Gc{\chi'}}(V)$}; \\
\node (a) {$D^b(\bT^\chi\operatorname{-gmod})$}; \& \node (b)
{$D^b(\bT^{\chi'}\operatorname{-gmod})$}; \\
};
\draw (a) -- (b) node[below,midway]{$B_{\mathcal{F}}\Lotimes_{A_\chi}-$}; 
\draw (d) -- (a) node[left,midway]{$\Ext_{D^b_{\Gc{\chi}}(V)}(\perv_{\chi}, -)$} ; 
\draw (e) -- (b) node[right,midway]{$\Ext_{D^b_{\Gc{\chi'}}(V)}(\perv_{\chi'}, -)$}; 
\draw (d) -- (e) node[above,midway]{$\mathcal{F}\star$}; 
}\]
\begin{lemma}
For any sum of shifts of semi-simple perverse sheaves $\mathcal{F}$ in $D^b_{P_{\chi'}\times P_{\chi}}(GL_n)$, the bimodule  $B_{\mathcal{F}}$ is sweet, that is, it is projective as a left module over $\bT^{\chi'}$ and as a right module over $\bT^\chi$.
\end{lemma}
\begin{proof}
Since the anti-automorphism of taking inverse switches left and right module structures, it's enough to prove this for the left module structure.  

Note that for any flavored multi-segment $\bQ$ for $\chi'$, we have that $\mathbf{IC}_{\bQ}$ is a summand of  $\perv_{\chi'}$ by Theorem \ref{thm:all-ICs}. Thus, $\Ext_{D^b_{\Gc{\chi'}}(V)}(\perv_{\chi'}, \mathbf{IC}_{\bQ})$ is a projective as a left module over $\bT^{\chi'}$; in fact it is of the form $\bT^{\chi'} e_{\bQ}$ for an idempotent projecting to $\mathbf{IC}_{\bQ}$ as a summand of $\perv_{\chi'}$.

Since $\mathcal{F}$ is semi-simple, by the Decomposition Theorem, the complex $\mathcal{F}\star \perv_{\chi}$ is a sum of shifts of simple perverse sheaves in $D^b_{\Gc{\chi'}}(V)$.   Thus,  $B_{\mathcal{F}}=\Ext_{D^b_{\Gc{\chi'}}(V)}(\perv_{\chi'}, \mathcal{F}\star \perv_{\chi})$ is a sum of projective  left $\bT^{\chi'}$-modules, and thus projective. 
\end{proof}

It follows immediately that the functor of tensor product with $B_{\mathcal{F}}$ is exact and:
\begin{lemma}\label{lem:BFG}
$B_{\mathcal{F}\star \mathcal{G}}\cong B_{\mathcal{F}}\Lotimes B_{\mathcal{G}}=B_{\mathcal{F}}\otimes B_{\mathcal{G}}$.
\end{lemma}

\begin{theorem}\label{thm:E-iso}
If $\chi'=\chi^{+i^a}$, then we have that $B_{\C_{\chi',\chi}}=\mathbb{E}_i^{(a)}$ and $B_{\C_{\chi,\chi'}}=\mathbb{F}_i^{(a)}$.
\end{theorem}
\begin{proof}
Of course, we only need to check one idempotent at a time.  That is, let $\Bi$ be a $\chi$-parabolic word, and $\Bi'$ a $\chi'$ parabolic.  We need only show that 
\[e(\Bi')\mathbb{E}_i^{(a)}e(\Bi)\cong e(\Bi')B_{\C_{\chi',\chi}}e(\Bi) =\Ext_{D^b_{\Gc{\chi}}(V)}(\perv_{\Bi'}, \C_{\chi',\chi}\star \perv_{\Bi}).\]
Note that this isomorphism is only correct up to grading shift due to the need to shift $\C_{\chi',\chi}$ to make it Verdier-self-dual.
Since $\C_{\chi',\chi}\star-$ is the composition of restriction of $\perv_\Bi$ to the $P_\chi\cap P_{\chi'}$-equivariant derived category, with the induction to the $P_{\chi'}$-equivariant derived category, we can use adjunction to show that 
\begin{equation}\label{eq:adjunction}
\Ext_{D^b_{\Gc{\chi}}(V)}(\perv_{\Bi'},  \C_{\chi',\chi}\star \perv_{\Bi})\cong \Ext_{D^b_{\Gc{\chi}\cap \Gc{\chi'}}(V)}(\perv_{\Bi'},\perv_{\Bi})\cong\Ext_{D^b_{B\times \Gz}(V)}(\perv_{\Bi'},\perv_{\Bi})^{S_{\chi',\chi}}
\end{equation}
Of course, by Theorem \ref{thm:T-iso} in the case where $P_\chi=B$, we have:
\begin{equation}\label{eq:B-case}
    \Ext_{D^b_{B\times \Gz}(V)}(\perv_{\Bi'},\perv_{\Bi})\cong e(\Bi')\bT e(\Bi).
\end{equation}
Thus, combining \eqref{eq:adjunction} and \eqref{eq:B-case}, we find that:
\[e(\Bi')\mathbb{E}_i^{(a)}e(\Bi)=(e(\Bi')\bT e(\Bi))^{S_{\chi',\chi}}=\Ext_{D^b_{B\times \Gz}(V)}(\perv_{\Bi'},\perv_{\Bi})^{S_{\chi',\chi}}=e(\Bi')B_{\C_{\chi',\chi}}e(\Bi).\qedhere\]
\end{proof}
Combining Lemma \ref{lem:BFG} and Theorem \ref{thm:E-iso}, we find that:
\begin{corollary}
We have a representation of the category $\tU$ sending $\mu_\chi \mapsto \bT^\chi\mmod$ where any 1-morphism $u$ acts by tensor product with the bimodule $B_{\Phi_{\mathsf{P}}(u)}$, and in particular, $\eE_i^{(a)},\eF_i^{(a)}$ act by the bimodules $\mathbb{E}_i^{(a)},\mathbb{F}_i^{(a)}$.
\end{corollary}
This generalizes \cite[Thm. 9.1]{KLSY}, which is the special case where $\rankp=2$, and avoids the long computation required by the direct proof given in that paper.

\subsection{Restriction to the flag variety}
\label{sec:restr-flag-vari}

Assume now that $v_1\leq  v_2\leq \cdots \leq v_n$.  In this case, the stack $V/\Gz$ contains an open subset where the compositions $f_{i;1}=f_i\cdots f_1$ are injective for all $i$.  On this open subset, \[\operatorname{image}(f_1)\subset \operatorname{image}(f_{2;1})\subset \cdots \subset \operatorname{image}(f_{n;1})\subset \C^n\] gives a partial flag.  As before, we let $\nu=(1,\dots, 1,2, \dots, 2,\dots)\in \Z^n$ be the unique increasing sequence with $i$ occurring $v_i-v_{i-1}$ times.  We can then describe our space of partial flags as exactly $GL_n/P_{\nu}.$ Thus we obtain a functor 
\[\Res\colon  D^{b,\operatorname{mix}}_{\Gc{\chi}}(V)\to D^{b,\operatorname{mix}}_{P_\chi}(GL_n/P_{\nu})\]
by restriction to this open set.  This is, of course, compatible with the natural action of $\Perv$.  

\begin{theorem}
We have an equivalence of categories \[\beta=\Ext_{D^b_{P_{\chi}}(GL_n/P_\nu)}(\Res(\perv_{\chi}), -)\colon D^{b,\operatorname{mix}}_{P_\chi}(GL_n/P_\nu)\cong D^b(\vT^\chi_\nu\gmod)\] such that we have a commutative diagram:
\[\tikz[->,thick]{
\matrix[row sep=18mm,column sep=30mm,ampersand replacement=\&]{
\node (d) {$D^{b,\operatorname{mix}}_{\Gc{\chi}}(V)$}; \& \node (e)
{$ D^{b,\operatorname{mix}}_{P_{\chi}}(G/P_{\nu})$}; \\
\node (a) {$D^b(\bT^\chi\operatorname{-gmod})$}; \& \node (b)
{$D^b(\vT^{\chi}\operatorname{-gmod})$}; \\
};
\draw (a) -- (b) node[below,midway]{$\vT^{\chi}\Lotimes_{\bT^\chi}-$}; 
\draw (d) -- (a) node[left,midway]{$\Ext(\perv_{\chi}, -)$} ; 
\draw (e) -- (b) node[right,midway]{$\beta$}; 
\draw (d) -- (e) node[above,midway]{$\Res$}; 
}\]
\end{theorem}
\begin{proof}
This is essentially a restatement of \cite[Cor. 5.4]{Webqui}, but let us be a bit more careful about the details of this point.  We have a map $\bT^\chi\to \Ext^*(\Res(\perv_{\chi}), \Res(\perv_{\chi}))$ induced by the functor $\Res$.  

If $\Bi$ is violating, then  $(i_1,1)\prec (\rankp,1)$.  For any point in $\Fl(\Bi)$, by assumption, we have $f_{\rankm;i_1}F_{\preceq(i_1,1)}=0$, and so one of $f_k$ is not injective. This shows that  $\Res(\mathscr{F}_{\Bi})=0$.
On the other hand, if $\Res(\mathbf{IC}_{\bQ})=0$, then $\bQ$ must be a multi-segment that includes a segment without $\rankp$, so the corresponding good word $\Bi_{\bQ}$ is violating.  Thus, $\Res(\mathbf{IC}_{\bQ})=0$ if and only if $\mathbf{IC}_{\bQ}$ is a summand of $\mathscr{F}_{\Bi}$.

The usual formalism of recollement shows that $D^{b,\operatorname{mix}}_{P_\chi}(GL_n/P_\nu)$ (after suitable enhancement) is a dg-quotient of $D^{b,\operatorname{mix}}_{\Gc{\chi}}(V)$ by the subcategory of objects killed by $\Res$; by the paragraph above, this is exactly the subcategory generated by $\mathscr{F}_{\Bi}$ for $\Bi$ violating.
Since $\bT^\chi$ is quasi-isomorphic to $\Ext^*(\perv_{\chi}, \perv_{\chi})$ as a formal dg-algebra, taking dg-quotient by the projectives $\bT^\chi e(\Bi)$ for $\Bi$ violating has the effect of simply modding out by the corresponding 2-sided ideal.  Thus, the map above induces an isomorphism 
\[\vT^\chi=\Ext^*_{D^{b,\operatorname{mix}}_{P_\chi}(GL_n/P_\nu)}(\Res(\perv_{\chi}), \Res(\perv_{\chi})).\] By dg-Morita theory, this completes the proof.
\end{proof}
Of course, we can also interpret $D^b_{P_{\chi}}(GL_n/P_\nu)$ as a Hom-category in the equivalent 2-categories $\Flag\cong \Perv$, and this identification intertwines the left regular action on $\Perv$ with the action we've discussed on $D^b_{P_{\chi}}(GL_n/P_\nu)$.  

To fix this equivalence algebraically, we need to describe the image of the identity.  This corresponds to the constant sheaf on $P_\nu/P_\nu\subset GL_n/P_{\nu}$.  Consider the flavored multi-segment where we take the segment $(i,\dots,m)$ exactly $v_i-v_{i-1}$ times, and flavor these with $i$.

The corresponding word is 
\[ \Bi_{\nu}=(m^{(v_1)},(m-1)^{(v_1)},\dots, 1^{(v_1)},m^{(v_2-v_1)},(m-1)^{(v_2-v_1)},\dots, 2^{(v_2-v_1)},\dots).\]

\begin{lemma}\label{lem:constant-sheaf}
We have an isomorphism $\Res(\mathscr{F}_{\Bi_\nu})\cong \C_{P_{\nu}/P_{\nu}}$.  
\end{lemma}
\begin{proof}
Consider the space $\C^{v_{i}}$.  If $f$ is injective on $\C^{v_j}$ for all $j<m$, then $f^{m-i}(\C^{v_{i}})$ is a $v_i$-dimensional subspace in $\C^n$.  Furthermore, this subspace must lie in the span of $b_{m,j}$ with $(m,j)\prec (i,v_i),$ since $F_{\preceq (i,v_i)}$ is a submodule containing $\C^{v_i}.$ In the word $\Bi_{\nu}$, we have $(m,1) \preceq \cdots \preceq (m,v_i)\prec (i,v_i)\prec (m,v_{i+1}).$  This shows that the image of $f^{m-i}$ must be exactly the $v_i$ dimensional space in the standard flag; this completes the proof.  
\end{proof}

\begin{corollary}\label{cor:Soergel-T}
The category of graded projective $\vT^\chi_\nu$-modules is equivalent to the category $\Hom_{\Flag}(\chi,\nu)$ of singular Soergel bimodules for the singularity $(\chi,\nu)$.  This equivalence is characterized by 
\begin{enumerate}
    \item sending the diagonal bimodule in $\Hom_{\Flag}(\nu,\nu)$ to the module $\vT^\nu_\nu e(\Bi_{\nu})$,
    \item intertwining the action of $\Flag\cong \Perv$ by ladder bimodules with the left regular action on singular Soergel bimodules.  
    \item intertwining the action of $\Flag$ by induction and restriction functors with the right regular action on singular Soergel bimodules.
\end{enumerate}
\end{corollary}
This generalizes the main theorem of \cite{KSred}.  Of course, the category of Soergel bimodules has a ``reversing functor'' where one swaps the left and right actions; this is an equivalence $\Hom_{\Flag}(\chi,\nu)\cong \Hom_{\Flag}(\nu,\chi)$ which is anti-monoidal.  Applying Corollary \ref{cor:Soergel-T}, we find that:
\begin{corollary}\label{cor:chi-nu}
We have a Morita equivalence between $\vT^\chi_\nu$ and $\vT^\nu_\chi$ which swaps the action of $\Flag$ via ladder bimodules with the action via induction and restriction functors.  
\end{corollary}

Of course, this Morita equivalence is realized by a bimodule over these two algebras.  It should be possible to capture this equivalence via a diagrammatic bimodule, but the author has not yet been successful in producing a workable description of it.  One important special case is that relating $\vT^\nu_\nu$ to itself.  This is, of course, true tautologically, but there are two ways of relating the ring $e(\Bi_{\nu})\vT^\nu_\nu e(\Bi_{\nu})$ to the ring $H^*_{P_{\nu}\times P_{\nu}}(P_\nu)$.  The one naturally arising from the left action of $P_{\nu}$ is given by polynomials on the red strands, which are symmetrized on each strand.  In the diagram, we insert symmetric polynomials in alphabets with sizes given by $v_1,v_2-v_1,\dots v_m-v_{m-1}$ at each point marked by a $*$ on the red strands:
\[   \tikz[baseline, xscale=2, yscale=.8]{
\draw[wei] (-2,-1) to[out=90,in=-90] node[midway,fill=white,circle, draw,thick]{$*$} node[at end,above,scale=.8]{$v_1$} (-2,1);
\draw[very thick] (-1.5,-1)to[out=90,in=-90] node[at start,below,scale=.8]{$(m-1)^{(v_1)}$}(-1.5,1);
\node at (-1,0) {$\cdots$};
\node at (1,0) {$\cdots$};
\node at (2,0) {$\cdots$};
\draw[very thick] (-.5,-1)to[out=90,in=-90] node[at start,below,scale=.8]{$1^{(v_1)}$}(-.5,1);
\draw[wei] (0,-1) to[out=90,in=-90]node[at end,above,scale=.8]{$v_2-v_1$} node[midway,fill=white,circle, draw,thick]{$*$} (0,1);
\draw[very thick] (.5,-1)to[out=90,in=-90] node[at start,below,scale=.8]{$(m-1)^{(v_2-v_1)}$}(.5,1);
\draw[very thick] (1.5,-1)to[out=90,in=-90] node[at start,below,scale=.8]{$2^{(v_2-v_1)}$}(1.5,1);
\draw[wei] (2.5,-1) to[out=90,in=-90] node[midway,fill=white,circle, draw,thick]{$*$} node[at end,above,scale=.8]{$v_m-v_{m-1}$} (2.5,1);
\draw[very thick] (3,-1)to[out=90,in=-90] node[at start,below,scale=.8]{$(m-1)^{(v_m-v_{m-1})}$}(3,1);
   }\]That arising from the right action is given by the same polynomials on the rightmost strand with each label:
\[   \tikz[baseline, xscale=2, yscale=.8]{
\draw[wei] (-2,-1) to[out=90,in=-90] node[at end,above,scale=.8]{$v_1$} (-2,1);
\draw[very thick] (-1.5,-1)to[out=90,in=-90] node[at start,below,scale=.8]{$(m-1)^{(v_1)}$}(-1.5,1);
\node at (-1,0) {$\cdots$};
\node at (1,0) {$\cdots$};
\node at (2,0) {$\cdots$};
\draw[very thick] (-.5,-1)to[out=90,in=-90] node[midway,fill=white,circle, draw,thick]{$*$} node[at start,below,scale=.8]{$1^{(v_1)}$}(-.5,1);
\draw[wei] (0,-1) to[out=90,in=-90]node[at end,above,scale=.8]{$v_2-v_1$}  (0,1);
\draw[very thick] (.5,-1)to[out=90,in=-90] node[at start,below,scale=.8]{$(m-1)^{(v_2-v_1)}$}(.5,1);
\draw[very thick] (1.5,-1)to[out=90,in=-90] node[midway,fill=white,circle, draw,thick]{$*$} node[at start,below,scale=.8]{$2^{(v_2-v_1)}$}(1.5,1);
\draw[wei] (2.5,-1) to[out=90,in=-90] node[at end,above,scale=.8]{$v_m-v_{m-1}$} (2.5,1);
\draw[very thick] (3,-1)to[out=90,in=-90] node[midway,fill=white,circle, draw,thick]{$*$} node[at start,below,scale=.8]{$(m-1)^{(v_m-v_{m-1})}$}(3,1);
   }\]
Geometrically, these correspond to polynomials in the Chern classes of tautological bundles: in the first case, of the associated bundles for the action of $P_{\nu}$ on simple subquotients of $\C^n$ under the left action of $P_{\nu}$; in the second case on $\C^{v_k}/f(\C^{v_{k-1}})$.  These are the same, since as discussed in the proof of Lemma \ref{lem:constant-sheaf}, the corresponding vector bundles are isomorphic.   
 \section{Gelfand-Tsetlin modules}

Let $U=U(\mathfrak{gl}_n)$.  As mentioned in the introduction, we let $\Gamma\subset U(\mathfrak{gl}_n)$\notation{$\Gamma$}{The Gelfand-Tsetlin subalgebra of $U$.} be the {\bf Gelfand-Tsetlin subalgebra} generated by the centers $Z_k=Z(U(\mathfrak{gl}_k))$ for $k=1,\dots, n$ where $\mathfrak{gl}_k\subset \mathfrak{gl}_n$ is embedded as the top left corner.   In this section, we'll discuss how $U(\mathfrak{gl}_n)$ and its representation theory relate to the algebras $\bT^\chi$ in the case where $v_i=i$ for $i=1,\dots, n$.  In this case,  $\Omega=\{(i,j) \mid 1\leq j\leq i\leq n\}$.

For the dominant weight $\chi$, we have a maximal ideal $\mathfrak{m}_\chi$ of $Z_n $ defined by the kernel of the action on the Verma module with highest weight
\[\chi-\rho=(\chi_1-1,\dots, \chi_n-n).\]

\begin{definition}
A {\bf Gelfand-Tsetlin module} is a finitely generated $U(\mathfrak{gl}_n)$-module on which the action of $\Gamma$ is locally finite. 
\end{definition}

Of course, by standard commutative algebra, if $M$ is a Gelfand-Tsetlin module then $M$ breaks up as a direct sum over the maximal ideals in $\MaxSpec(\Gamma)$.  Every maximal ideal of $\Gamma$ is generated by maximal ideals in $Z_k$ for each $k$, which we index with an unordered $k$-tuple $\la_k=(\la_{k,1},\dots, \la_{k,k})$; as above, we match this with the  maximal ideal in $Z_k$ acts trivially on the Verma module over $U(\mathfrak{gl}_k)$ with highest weight $\la_k-\rho_k=(\la_{k,1}-1,\dots, \la_{k,k}-k).$  Let $\mGT_{\la}\subset \Gamma$ be the corresponding maximal ideal.
 Thus, for any Gelfand-Tsetlin module, we have a decomposition \[M=\bigoplus_{\la \in \MaxSpec(\Gamma)}\Wei_\la(M)\hspace{5mm}\text{ where }\hspace{5mm}
 \Wei_\la(M)=\{m\in M\mid \mGT_\la^Nm=0 \text{ for } N\gg 0\}.\notation{$\Wei_\la$}{The weight functor $ \Wei_\la(M)=\{m\in M\mid \mGT_\la^Nm=0 \text{ for } N\gg 0\}$.}\]

\begin{remark}
To help the reader fix notation in their mind, this means that the maximal ideals that appear in finite-dimensional modules with $\chi\in \Z^n$ are given by $\la_k\in \Z^k$ with $\la_n=\chi$ that satisfy
\begin{equation}
\la_{k+1,1}\leq \la_{k,1}< \la_{k+1,2}\leq \la_{k,2}< \cdots <\la_{k+1,k}\leq \la_{k,k}<\la_{k+1,k+1} \label{eq:GT}
\end{equation}  for $k=1,\dots, n-1$.  Readers will recognize this as the condition that  $\la_k-\rho_k$ form a Gelfand-Tsetlin pattern.  Under this correspondence,  the trivial module gives $\la_k=(1,2,\dots, k).$ 
\end{remark} 

\begin{definition}
Let $\MaxSpec_{\Z}(\Gamma)\subset \MaxSpec(\Gamma)$ be the subset where $\la_{k,i}\in \Z$ for all $(k,i)\in \Omega$, and $\MaxSpec_{\Z,\chi}(\Gamma)\subset\MaxSpec_{\Z}(\Gamma)$  the subset where $\la_{n,i}=\chi_i$ for all $i$.  An {\bf integral Gelfand-Tsetlin module} is one where $\Wei_\la(M)=0$ if $\la\notin\MaxSpec_{\Z}(\Gamma)$.  

 Let \notation{$\mathcal{GT}_\chi$}{The category of integral Gelfand-Tsetlin modules with central character $\chi$.}$\mathcal{GT}_\chi$ be the category of integral Gelfand-Tsetlin modules over $U(\mathfrak{gl}_n)$ on which the ideal $\mathfrak{m}_\chi\subset Z_n$ acts nilpotently, i.e. those where $\Wei_\la(M)=0$ if $\la\notin\MaxSpec_{\Z,\chi}(\Gamma)$. 
\end{definition}

\begin{lemma}\label{lem:GT-action}
Let $V,M$ be $U(\mathfrak{gl}_n)$-modules with $M$ Gelfand-Tsetlin  and $V$ finite-dimensional.  The module $V\otimes M$ is Gelfand-Tsetlin, and if $M\in \mathcal{GT}_\chi$ then $V\otimes M$ is a sum of objects in $\mathcal{GT}_{\chi+\mu}$ for $\mu$ a weight of $V$.  
\end{lemma}
\begin{proof}
Since $\Gamma$ is commutative, it is enough to show local finiteness separately under a set of generating subalgebras.  That is, it suffices to check that $V\otimes M$ is locally finite under $Z_k=Z(U(\mathfrak{gl}_k))$ for all $k\leq n$.  This follows from \cite[Cor. 2.6(ii)]{BeGe}, which shows that $V\otimes M$ is locally finite under $Z_k$ whenever $M$ is locally finite under $Z_k$.  

By \cite[Thm. 2.5(ii)]{BeGe}, the set of modules with integral central character is closed under tensor product, so if $M$ is integral Gelfand-Tsetlin, then $V\otimes M$ is as well.  Finally, applying \cite[Thm. 2.5(ii)]{BeGe} again implies the statement on characters under $Z_n$.    
\end{proof}

In particular, we have functors $E(M)=\C^n\otimes M$ and $F(M)=(\C^n)^*\otimes M$\notation{$E(M),F(M)$}{The functors $E(M)=\C^n\otimes M$ and $F(M)=(\C^n)^*\otimes M$.} on the category of Gelfand-Tsetlin modules with integral central character $\chi\in \Z^n$.
Lemma \ref{lem:GT-action} shows that:
\begin{corollary}
If $M\in \mathcal{GT}_\chi$, then $E(M)$ is a sum of objects in $\mathcal{GT}_{\chi^{+i}}$ for $i\in \{\chi_1,\dots, \chi_n\}$ and $F(M)$  is a sum of objects in $\mathcal{GT}_{\chi^{-i}}$ for $i\in \{\chi_1-1,\dots, \chi_n-1\}$ 
\end{corollary}
Since decomposing according to the spectrum of $Z_n$ is functorial, we can define functors $E_i$ (resp.\ $F_i$) such that for $M\in\mathcal{GT}_{\chi}$, we have that $E_i(M)\subset E(M)$ (resp. $F_i(M)\subset F(M)$)\notation{$E_i(M),F_i(M)$}{The maximal summands of $E,F$ that send $\mathcal{GT}_{\chi}$ into the category $\mathcal{GT}_{\chi^{\pm i}}$.} is the unique maximal summand that lies in the category $\mathcal{GT}_{\chi^{\pm i}}$.  These are also the generalized weight spaces of the Casimir operator on $V\otimes M$ (this follows, for example, by \cite[Lem. 4.5]{BSW3}).  This gives a decomposition of these functors acting on integral Gelfand-Tsetlin modules:
\[E=\bigoplus_{i\in \Z} E_i\qquad F=\bigoplus_{i\in \Z} F_i.\] 
It is well-known that the functors $E_i, F_i$ define categorical actions (for example, see the discussion of category $\mathcal{O}$ in \cite[\S 7.5]{CR04}) on various categories of $\mathfrak{gl}_n$-modules, but due to artifacts of the proofs, it is not obvious that these apply to the category of Gelfand-Tsetlin modules.  Recent work of Brundan, Savage and the author shows how to unify these proofs:
\begin{lemma}
The functors $E$ and $F$ define an action of the {\bf level 0 Heisenberg category} of \cite{BruHei} (also called the {\bf affine oriented Brauer category} in \cite{BCNR}). The functors $E_i$ and $F_i$ for $i\in \Z$ give an induced categorical Kac-Moody action of $\mathfrak{sl}_\infty$, and these actions are related by \cite[Th. A]{BSW3}.
\end{lemma}
Note that the formulas of \cite[Th. 4.11]{BSW3} give explicit formulas for the action of $\tU$ in terms of swapping factors in tensor products and the Casimir element, but these formulas are not needed for our purposes.  

Let us note another perspective on this result.  Recall, that we call a $U$-$U$ bimodule {\bf Harish-Chandra} if it is locally finite for the adjoint action and {\bf pro-Harish-Chandra} if it satisfies this property in the topological sense. 

The most important example of a Harish-Chandra bimodule is the tensor product $V\otimes U$ for $V$ a finite-dimensional $U$-module, with the right action by right multiplication on the second factor, and the left action via the coproduct. That is, for $v\in \C^n,u\in U, X_1,X_2\in \mathfrak{gl}_n$, we have:
\[X_1\cdot (v\otimes u)\cdot X_2=X_1v\otimes uX_2+v\otimes X_1 u X_2.\]  The resulting bimodule structure on $V\otimes U$  is Harish-Chandra.   Thus, the functor $E_i\colon \mathcal{GT}_{\chi}\to \mathcal{GT}_{\chi^{+i}}$ is given by tensor product with a pro-Harish-Chandra $U$-$U$ bimodule $\mathscr{E}_i$, formed by completing the left and right actions of $Z_n$ on $\C^n\otimes M$ with respect to the appropriate maximal ideal.  Let $\mathsf{HC}$ be the 2-category such that:
\begin{itemize}
    \item objects are dominant integral weights $\chi=(\chi_1\leq \cdots \leq \chi_n)$,
    \item 1-morphisms $\chi\to \chi'$ are sums of pro-Harish-Chandra bimodules where the maximal ideal $\mathfrak{m}_{\chi'}\subset Z_n$ acts topologically nilpotently on the left and $\mathfrak{m}_\chi\subset Z_n$ acts topologically nilpotently on the right.  
    \item 2-morphisms are homomorphisms of bimodules. 
\end{itemize}
By \cite[Th. A]{BSW3}, we have a 2-functor $\tU^*\to \mathsf{HC}$ sending $\EuScript{E}_i\mapsto \mathscr{E}_i$ and similarly with $\mathscr{F}$.  Note that Soergel sketches the construction of a functor $\mathsf{Flag}\to \mathsf{HC}$ which can be used to relate this to Khovanov and Lauda's action in the final paragraph of \cite{Soe92}.

\subsection{Presentation of the category \texorpdfstring{$\mathcal{GT}_{\chi}$}{GT}}

A choice of integral $\la\in \MaxSpec_{\Z,\chi}(\Gamma)$ can be used to give a total  preorder $\prec$ on the set $\Omega$ as in previous sections for the dimension vector $v_i=i$ for $\rankp=n$.  We set $(i,k) \preceq (j,\ell)$ if and only if  $\la_{i,k}< \la_{j,\ell}$ or $\la_{i,k}= \la_{j,\ell}$ and $i\geq j$.  Note that the resulting preorder satisfies \eqref{eq:order-1}  by assumption and  \eqref{eq:order-2} since if $(i,k) \approx (j,\ell)$, then we must have $\la_{i,k}=\la_{j,\ell}$ and $i=j$.
\begin{definition}
We let $\Bi(\la)$ define the corresponding word in $[1,m+1]$ and $e'(\la)$ be the corresponding idempotent in $\bT^\chi$.  
\end{definition}

Note that \cite[Prop. 5.4]{WebGT} shows that any maximal ideals giving the same preorder on $\Omega$ give isomorphic functors on $\mathcal{GT}_\chi$.
If we change the order by swapping pairs $(i,k) $ and $(j,\ell)$ this is called a {\bf neutral swap}, and more generally neutral swaps don't change the functor given by two maximal ideals.  For example, two Gelfand-Tsetlin patterns with the same $\chi$ don't necessarily give the same order, but they differ by neutral swaps, and indeed they give isomorphic functors.  

\begin{example}
The action of $\Gamma$ on the trivial representation gives the weight $\la_{i,k}=k$.  Thus, the resulting word is 
\[(n,n-1,\dots, 2,1,n, n-1, \dots, 3,2, n,n-1,\dots, 3, \dots n,n-1,n).\]
As mentioned before, every Gelfand-Tsetlin pattern gives an ordering that differs from this one by neutral swaps.

On the other hand, a weight like $\la_3=(1,2,3), \la_2=(4,4), \la_1=(1)$ gives $(3,1,3,3,2^{(2)})$, whereas if $\la_2=(4,5)$, then we have $(3,1,3,3,2,2).$ This is an example of a maximal ideal of the Gelfand-Tsetlin subalgebra that appears in no finite-dimensional module.
\end{example}

Using the techniques of \cite{WebGT, WebSD, KTWWYO}, we can give an algebraic presentation of the category $\mathcal{GT}_{\chi}$.   This is accomplished by presenting the algebra of natural transformations of the functors $\Wei_\la.$
Let $\widehat{\bT}_\chi$ be the completion of the algebra $\bT_{\chi}$ with respect to its grading.  See \cite[\S 2]{WebBKnote} for more discussion of the induced topology (in particular, the continuity of multiplication in it).  
\begin{theorem}\label{thm:GT-iso}
We have an isomorphism compatible with multiplication
\[e'(\la') \widehat{\bT}_\chi e'(\la)\cong \Hom(\Wei_\la,\Wei_{\la'}).\]
\end{theorem}
This is extremely closely related to \cite[Thm. 5.2]{KTWWYO}, but slightly different, since that result is for Gelfand-Tsetlin modules with honest central character, which corresponds in $\widehat{\bT}_\chi $ to setting red dots to 0.  Under this isomorphism, the red dots give the nilpotent parts of certain (complicated) elements of the center of $U(\mathfrak{gl}_n).$ In \cite[Lem. 3.3 \& Prop. 3.4]{SilverthorneW}, we give a more direct algebraic proof of this theorem.
\begin{proof}
This follows from \cite[Thm. 4.4]{WebGT}.  The space $\Hom(\Wei_\la,\Wei_{\la'})$ is exactly the bimodule ${}_{\la'}\hat{U}_{\la}=\varprojlim U/(\mGT_{\la'}^NU+U\mGT^N_{\la})$, and by  \cite[(4.5)]{WebGT}, this is exactly the completion with respect to grading of $\Ext(\perv_{\Bi(\la)},\perv_{\Bi(\la')}).$  On the other hand, by Theorem \ref{thm:T-iso}, we have an isomorphism $\Ext(\perv_{\Bi(\la)},\perv_{\Bi(\la')})\cong e'(\la') {\bT}_\chi e'(\la).$  After applying completion, this induces an isomorphism  ${}_{\la'}\hat{U}_{\la}\cong e'(\la') \widehat{\bT}_\chi e'(\la)$.  Both of these isomorphisms are compatible with multiplication, so this completes the proof.   
\end{proof}

\begin{definition}\label{def:wgmod}
Let $\bT^\chi\wgmod$ be the category of finite-dimensional $\bT^\chi$-modules $M$ which are weakly gradable, that is, $M$ has a finite filtration with gradable subquotients.  

These are precisely the modules that extend to finite-dimensional modules over the completion $\widehat{\bT}_\chi$ with the discrete topology.  
\end{definition}

\begin{corollary}\label{cor:GT-equiv}
We have an equivalence of categories $\boldsymbol{\Theta}\colon \mathcal{GT}_\chi\cong {\bT}_\chi\wgmod$ such that $e(\la)\boldsymbol{\Theta}(M)=\Wei_{\la} (M)$ for all integral maximal ideals $\la\in \MaxSpec_{\Z,\chi} (\Gamma)$.  
\end{corollary}
\begin{proof}
As $\la$ runs over $\MaxSpec_{\Z,\chi} (\Gamma)$, we only get finitely many different idempotents as $e(\la)$;  if we let $\mathsf{S}$ be a set of maximal ideals that contains one representative of each idempotent, then this is {\bf complete} in the sense of \cite[Def. 2.24]{WebGT}.  Let $\bar{e}$ be the sum of these idempotents.  Then \cite[Def. 2.23]{WebGT} shows that the functor $\oplus_{\la\in \mathsf{S}}\Wei_{\la}$ is an equivalence of categories $\mathcal{GT}_\chi\cong \bar{e}\widehat{\bT}_\chi\bar{e}\wgmod$ with the desired properties.  As argued in \cite[Prop. 5.12]{WebGT}, no simple object in $\widehat{\bT}_\chi\mmod$ is killed by $\bar{e}$, so $\bar{e}$ defines a Morita equivalence giving the desired result.  

Note that if $\chi_i \ll \chi_{i+1}$, then $\bar{e}$ is the identity in $\bT^\chi$ and this last step is not needed.  The general case can be deduced from this one using translation functors, but we appeal to \cite[Prop. 5.12]{WebGT} since we have not yet discussed compatibility with translation functors.
\end{proof}
Together with Theorem \ref{thm:T-iso}, 
this establishes Theorem \ref{thm:main}. 

Let us record how this equivalence matches  various properties of interest for Gelfand-Tsetlin modules:
\begin{lemma}
	A Gelfand-Tsetlin module $M$ satisfies:
	\begin{enumerate}
		\item $U(\mathfrak{n})\subset U(\mathfrak{gl}_n)$ acts locally finitely if and only if $\Theta(M)$ factors through $\vT^{\chi}_\nu$.
		\item The Gelfand-Tsetlin subalgebra acts semi-simply if and only if all dots, red and black, act trivially.
		\item The Cartan $\mathfrak{h}$ acts semi-simply if and only if the sum $h_{i,1}$ of all dots on strands with label $i$ acts trivially, for all $i$.
		\item The center $Z_n$ acts semi-simply if and only if all positive degree homogeneous polynomials in the red dots act trivially.
	\end{enumerate}
\end{lemma}

In particular, the modules over $\bT^\chi_\nu$ which factor through the quotient by all red dots and by violating strands is the category $\cO'_\chi$ of modules locally finite under $U(\mathfrak{n})$ and semi-simple for the action of the center with central character $\chi$; such a module is automatically a generalized weight module, but not necessarily an honest weight module, and so not necessarily in category $\cO$.

On the other hand, the modules that factor through the quotient of $\vT^{\chi}_\nu$ by the symmetric polynomials $h_{i,1}$ for all $i$, are those which lie in the usual category $\cO$.  Note that we showed in \cite[Cor. 9.10]{Webmerged} that $\bvT^{\nu}_\chi\mmod$ was equivalent to $\cO_\chi$ with the translation functors matched with the induction-restriction action.  Thus, this is related to our result using the Morita equivalence of Corollary \ref{cor:chi-nu} between  $\bvT^{\nu}_\chi$ and this quotient of $\vT^\chi_\nu$.  

\subsection{Translation functors}
Now, let us consider how this perspective interacts with translation functors.  This requires studying the bimodule $\C^n\otimes U$ 
\begin{lemma}
As a $U\operatorname{-}U$-bimodule, $\C^n\otimes U$ is isomorphic to the sub-bimodule $U'$ in $U(\mathfrak{gl}_{n+1})$ generated by $e_n$.
\end{lemma}
\begin{proof}
Both these bimodules have an induced adjoint $\mathfrak{gl}_n$-module structure (which exists for any Hopf algebra). The subspace $\C^n\subset \C^n\otimes U$ is a submodule with this module structure.  Similarly, $e_n$ generates a copy of $\C^{n}$ inside the adjoint module over $\mathfrak{gl}_{n+1}$ restricted to $\mathfrak{gl}_n$.  Sending $e_n\mapsto (0,\dots, 0,1)$ induces a unique isomorphism $\psi$ between these subspaces.  

In any $U\operatorname{-}U$ bimodule, a generating subspace for the bimodule action which is closed under the adjoint action is already generating for the right action.  The module $\C^n\otimes U$ is freely generated by $\C^n$ as a right $U$-module by definition, and the same is true for $U'$ by the PBW theorem.  Thus, $\psi$ canonically extends to a right-module homomorphism, which is also a left-module homomorphism since $\psi$ is adjoint equivariant.  
\end{proof}
Fix $\chi\in \MaxSpec(Z_n)$, and assume that $\chi^{+i}$ exists.  Consider  \[\la\in \MaxSpec_{\Z,\chi}(\Gamma)\qquad \la'\in \MaxSpec_{\Z,\chi^{+i}}(\Gamma).\] 
\begin{theorem} We have an isomorphism compatible with bimodule structure:
\[e'(\la') \mathbb{E}_i e'(\la)=\varprojlim U'/(\mGT_{\la'}^NU'+U'\mGT_{\la}^N)\]
\end{theorem}
\begin{proof}
Note that $e_n$ commutes with $Z_k$ for $k<n$, embedded in $U(\mathfrak{gl}_{n+1})$ via the standard inclusion.  Thus, $e_n$ only has non-zero image if $\la_{i,k}=\la'_{i,k}$ for all $i<n$.  We therefore can describe the desired map by specifying the image of $e_n$ in this case.  

As in \cite{WebSD, KTWWYO}, we can read this off by carefully studying the polynomial representation of $U(\mathfrak{gl}_{n+1})$, induced by its realization as an orthogonal Gelfand-Zetlin algebra\footnote{Apologies for the inconsistent spelling of \foreignlanguage{russian}{Цетлин.  }We use a``Z'' here to match \cite{mazorchukOGZ}.}.  This was originally given in \cite{mazorchukOGZ}, though it might be easier to follow the notation in \cite[\S 5.1]{WebGT}.  Let $\Gamma_{(n+1)}$ be the Gelfand-Tsetlin subalgebra of $U(\mathfrak{gl}_{n+1})$, which we identify with the space of symmetric polynomials in the alphabets $\{x_{i,1},\dots, x_{i,i}\}$ for $i=1,\dots, n+1$.

 Let $\varphi_{i,j}$ be the translation on all polynomials in these alphabets
satisfying \[\varphi_{i,j}(x_{k,\ell})=(x_{k,\ell}+\delta_{ik}\delta_{j\ell})
  \varphi_{i,j}.\]
  We then have a representation of $\mathfrak{gl}_{n+1}$ where $E_i$ and $F_i$ act by the operators
\[X^\pm_i=\mp\sum_{j=1}^{i}\frac{\displaystyle\prod_{k=1}^{v_{i\pm 1}}
    (x_{i,j}-x_{i\pm 1,k})}{\displaystyle\prod_{k\neq j}
    (x_{i,j}-x_{i,k})}\varphi_{i,j}^{\pm 1}\]
    
Thus, $U'$ embeds into the bimodule of operators $\Gamma_{(n+1)}$ when this space is made into a module over $U(\mathfrak{gl}_{n})$ as the subspace generated by the operator 
\[X^+_n=-\sum_{j=1}^{n}\frac{\displaystyle\prod_{k=1}^{{n+ 1}}
    (x_{n,j}-x_{n+ 1,k})}{\displaystyle\prod_{k\neq j}
    (x_{n,j}-x_{n,k})}\varphi_{n,j}\]
    Consider $P'=\varprojlim U'/U'\mGT_{\la}^N$.  This is a left module over $U$, and clearly pro-Gelfand-Tsetlin, since for any element of $u\in U'$, the sub-bimodule $\Gamma u\Gamma$ lies in the left $\Gamma$ module generated by finitely many translations (this is a version of the Harish-Chandra property discussed in \cite{FOD}). This means that $\Gamma u\Gamma\otimes_{\Gamma}\Gamma/\mGT_{\la}^N$ is finite-dimensional so $U'/U'\mGT_{\la}^N$ is Gelfand-Tsetlin.  Thus $P'$ decomposes into pieces $P'=\oplus_{i\in \Z} P_i'$ on which $Z_n$ acts by $\chi^{+i}$ for $i$ ranging over the values of $\chi_g$; note that $P_i'=E_i(\varprojlim U/U\mGT_{\la}^N)$.

    Consider the image of $X^+_n$ in $P'_i$.  If we let $I_i=\{j\mid \la_{n,j}=i\}$, then we have that this image is given by 
    \[\sum_{j\in I_{i+1}}  \frac{\displaystyle\prod_{k=1}^{{n+ 1}}
    (x_{n,j}-x_{n+ 1,k})}{\displaystyle\prod_{k\in [1,n]\setminus \{j\}}
    (x_{n,j}-x_{n,k})} \varphi_{n,j}\]
Note that if $k\notin I_{i+1}$ then the corresponding factor in the denominator is invertible in the local ring $\varprojlim \Gamma_{n}/\Gamma_{n}\mGT_{\la'}^N$, using the geometric series
\[\frac{1}{x_{n,j}-x_{n,k}}=\frac{1}{\la_{n,j}'-\la'_{n,k}}+\frac{x_{n,k}-\la'_{n,k}-x_{n,j}+\la'_{n,j}}{(\la_{n,j}'-\la'_{n,k})^2}+\frac{(x_{n,k}-\la'_{n,k}-x_{n,j}+\la'_{n,j})^2}{(\la_{n,j}'-\la'_{n,k})^3}+\cdots\]
    Thus, we can write the image of $X^+_n$ as a sum of operators of the form \[X_{i}^{(p,s)}=-\sum_{j\in I_{i+1}}  \frac{\displaystyle\prod_{k=1}^{{n+ 1}}
    (x_{n,j}-x_{n+ 1,k})}{\displaystyle\prod_{p\in I_{i+1}\setminus \{j\}}
    (x_{n,j}-x_{n,p})} (x_{n,j}-i-1)^sp(x_{n,1},\dots, x_{n,n})\varphi_{n,j}\] where $p$ is a polynomial symmetric in the alphabets $I_i$.
    
    We can identify $p$ with a polynomial in the red dots at the bottom of the diagram, with the shifted variables $\{x_{n,j}-i\}_{j\in I_i}$ sent to the red dots on the corresponding strand with $i=\chi_j$.  
    We then send $X_{i}^{(p)}$ to the element of $\mathbb{E}_i$ where we place the polynomial $p$ applied to the red dots at the bottom of the diagram and  ``sandwich'' $s$ red dots on the middle strand of the ladder, with all black strands are straight vertical at positions determined by $\la$.  That is,
    \begin{equation}
        X_{i}^{(p,s)}\mapsto\,- \tikz[baseline, yscale=1.2]{\draw[wei] (-2,-1)--(-2,1);\draw[very thick] (-1.5,-1)--(-1.5,1);\draw[very thick] (-1,-1)--(-1,1);\draw[very thick] (0,-1)--(0,1);\draw[very thick] (.5,-1)--(.5,1);
    \draw[wei] (-.5,-.35) to[out=90,in=-90] node [red, midway, circle, fill=red, inner sep=3pt,label=above:{$s$}]{} (1,.85);\draw[wei] (1,-1)--(1,1); \draw[wei] (-.5,-1)--(-.5,1);\draw[very thick] (1.5,-1)--(1.5,1);\draw[wei] (2,-1)--(2,1); \node[red,draw=black, fill=white, inner xsep=60pt] at (0,-.6){$p$};} \label{eq:XPS}
    \end{equation}
We claim that this defines a map    
\[\varprojlim U'/(\mGT_{\la'}^NU'+U'\mGT_{\la}^N)\to e'(\la') \mathbb{E}_i e'(\la)\]

First, let us match the completion $\K[[Y_1,\dots,Y_N]]e'(\la)$ with the completed polynomial ring $\varprojlim \Gamma/\Gamma\mGT_\la^N$ by matching $Y_r$ with $x_{p,j}-\la_{p,j}$ if the $r$th strand from the left is the $j$th (from the left) with label $p$; let us simplify notation by writing $Y_{p,j}:=Y_r$ in this case.  Under this isomorphism, $X_i^{(p,s)}$ transforms to
\[\mathscr{X}_{i}^{(p,s)}=-\sum_{j\in I_{i+1}}  \frac{\displaystyle\prod_{k=1}^{{n+ 1}}
    (Y_{n,j}+\la_{n,j}-x_{n+ 1,k})}{\displaystyle\prod_{p\in I_{i+1}\setminus \{j\}}
    (Y_{n,j}-Y_{n,p})} (Y_{n,j})^sp(Y_{n,1}+\la_{n,1},\dots, Y_{n,n}+\la_{n,n})\varphi_{n,j}\]
This is exactly $C=\prod_{k=1}^{{n+ 1}}
    (Y_{n,j}+\la_{n,j}-x_{n+ 1,k})$ times the action in the polynomial representation of $\mathbb{E}_i$ given the diagram on the left-hand side of \eqref{eq:XPS} in Lemma \ref{lem:E-rep}.  Multiplication by $C$ commutes with the left and right actions on this bimodule, so this shows that the representations of bimodules match.
    
    This shows that \eqref{eq:XPS} defines a homomorphism, which must be injective by the faithfulness of both representations.  Obviously, every bare ladder is in the image of this map, so it must be surjective as well.
\end{proof}
\begin{corollary}
 The equivalence of Corollary \ref{cor:GT-equiv} intertwines the action of the translation functors $E_i$ on $\mathcal{GT}_\chi$ and the bimodules $\mathbb{E}_i$ on ${\bT}_\chi\mmod$.
\end{corollary}
Together with Theorem \ref{thm:E-iso}, this establishes Theorem \ref{thm:action}, completing the main results of the paper.
 \appendix
\section{The classification of orbits}\label{sec:orbits}
\centerline{by Jerry Guan and Ben Webster}
\bigskip 

In this appendix, we discuss the classification of orbits of $P_\chi$ on $V$, using the notation introduced earlier.  

Recall the definition of a $\chi$-flavored multi-segment from Definition \ref{def:segments}.
A {\bf subsegment} of $\Omega$ is a subset $S$ such that $S=\{ (k,j_k),(k+1,j_{k+1}),\dots, (\ell,j_{\ell})\} $ for a segment $(k,k+1,\dots, \ell)$.  A subsegment  carries a canonical charge: if $\ell< \rankp$, this requires no new information, and if $\ell=\rankp$, then we use $\chi_{j_{\rankp}}$ as our flavor.
A {\bf segmentation} of the set $\Omega$ is a partition of $\Omega$ into subsegments.  Every segmentation gives a choice of flavored multisegment.

\begin{definition}
For each segmentation $\Sigma$ of $\Omega$, there is a canonical map $f_{\Sigma}\colon \C^\Omega\to \C^{\Omega}$ sending $f(b_{k,j_k})=b_{k+1,j_{k+1}}$ when $(k,j_{k})$ and $(k+1,j_{k+1})$ lie in a subsegment together, and $f(b_{k,j})=0$ if there is no element $(k+1,j')$ in the same subsegment.
\end{definition}

\begin{lemma}\label{lem:orbits}
Every $\Gc{\chi}$-orbit on $V$ contains $f_{\Sigma}$ for some segmentation $\Sigma$, and $f_{\Sigma}$ and $f_{\Sigma'}$ are in the same orbit if and only if they have the same $\chi$-flavored multisegment.  Thus, we have a bijection between $\Gc{\chi}$-orbits and  $\chi$-flavored multi-segments with the corresponding dimension vector. 
\end{lemma}
Note that we can think of this as a generalization of Gabriel's theorem on the representations of type A quivers:  if $v_{\rankm}=0$, then orbits correspond to isomorphism types of quiver representations, which are classified by multi-segments.

\begin{proof}
Obviously, if two segmentations $\Sigma, \Sigma'$ correspond to the same $\chi$-flavored multisegment, then there is a permutation $\sigma$ of $\Omega$ with $\sigma(\Sigma)=\Sigma'$ which preserves the first index, such that if $\sigma((\rankp,i))=(\rankp,j)$ then $\chi_i=\chi_j$.  The induced linear map $\C^{\Omega}\to \C^\Omega$ lies in  $\Gc{\chi}$ and witnesses the fact that these are in the same orbit.

Now, we turn to showing that if $\Sigma$ and ${\Sigma'}$ are segmentations where $f_\Sigma$ and $f_{\Sigma'}$ are in the same orbit, then the corresponding $\chi$-flavored multisegment is the same.

We'll prove this by induction on $m$.  As before, let $f_{i;j}=f_{i-1}\cdots f_j\colon \C^{v_j}\to \C^{v_i}$. Consider the map $f_{\rankp;1}\colon \C^{v_1}\to \C^{v_{\rankp}}$.  The space $\C^{v_{\rankp}}$ has a unique finest partial flag that is invariant under $P_\chi$.  This is of the form $0\subset F_{p_1}\subset \cdots \subset F_{p_\ell}$ where $p_k$ ranges over the integers which appear as values $\chi_i$. The dimension $\dim F_{p_k}/F_{p_{k-1}}=g_k$ is the number of indices $i$ such that $p_k=\chi_i$.
Thus, consider the flag \[0\subset \ker f_{\rankp;1}\subset f_{\rankp;1}^{-1}(F_{p_1})\subset f_{\rankp;1}^{-1}(F_{p_2})\subset \cdots\subset  f_{\rankp;1}^{-1}(F_{p_\ell})\]
For $f_\Sigma$, obviously, $\dim f_{\rankp;1}^{-1}(F_{p_1})/\ker f_{\rankp;1}$ is the number of segments of the form $(1,\dots, \rankp)$ of charge $p_1$, and $\dim f_{\rankp;1}^{-1}(F_{p_k})/f_{\rankp;1}^{-1}(F_{p_{k-1}})$ is the number of charge $p_k$.  Since these dimensions are the same on orbits of $\Gc{\chi}$, the segmentations $\Sigma$ and ${\Sigma'}$ must have the same number of segments of the form $(1,\dots, \rankp)$ of each given charge.

Similarly, the rank of the map $f_{\rankm;1}\colon \ker f_{\rankp;1}\subset \C^{v_1}\to \C^{v_{\rankm}}$ is the number of segments of the form $(1,\dots, \rankm).$  More generally, the rank of the map $f_{k;1}\colon \ker f_{k+1;1}\subset \C^{v_1}\to \C^{v_{k}}$ gives the number of segments of the form $(1,\dots, k)$.

This shows that each flavored segment containing a 1 appears with the same multiplicity in the multisegments for $\Sigma$ and $\Sigma'$.   Thus, if $w,w'\subset \C^{\Omega}$ are the subrepresentations generated by $\C^{v_1}$ under $f_\Sigma$ and $f_{\Sigma'}$, the induced representations on $\C^{\Omega}/W$ and $\C^{\Omega}/W'$ are of the form $f_{\bar{\Sigma}}$ and $f_{\bar{\Sigma}'}$ for $\bar{\Sigma}$ and $\bar{\Sigma}'$ segmentations obtained by throwing out all subsegments containing a $(1,k)$ (and reindexing by decreasing all indices).  By induction, $\bar{\Sigma}$ and $\bar{\Sigma}'$ have the same flavored multisegment, so the same is true of $\Sigma$ and $\Sigma'$.  

This shows that we have an injective map from flavored multi-segments with the right dimension vector to the set of orbits.  Now we need to show that this map is also surjective.

Consider the flag 
\begin{equation*}
    0\subset\ker f_{2;1}\subset \cdots \subset \ker f_{\rankp;1}\subset f_{\rankp;1}^{-1}(F_{p_1}) \subset f_{\rankp;1}^{-1}(F_{p_2})\subset \cdots\subset  f_{\rankp;1}^{-1}(F_{p_\ell})\subset \C^{v_1}.
\end{equation*}
Choose a ordered basis of $\C^{v_1}$ compatible with this flag, and let $g_1\in GL(\C^{v_1})$ be a matrix mapping this basis to   $b_{1,1},\dots, b_{1,v_1}.$ By construction, the non-zero images of this set under $f_{k;1}$ form a linearly independent set, which span $\operatorname{image}(f_{k;1})$ compatibly with the intersection with the flag 
\begin{equation}
    0\subset\ker f_{k+1;k}\subset \cdots \subset \ker f_{\rankp;k}\subset f_{\rankp;k}^{-1}(F_{p_1})\subset \cdots\subset  f_{\rankp;k}^{-1}(F_{p_\ell})\subset \C^{v_{k}}.\label{eq:big-flag}
\end{equation}
Thus, we can choose a basis of $\ker f_{3;2}\subset \C^{v_2}$ whose union with the non-zero images $f(g_1^{-1}b_{1,*})$ is still linearly independent; we then in turn extend this in turn to bases of \[\ker f_{4;2}\subset \ker f_{5;2}\subset \dots  f_{\rankp;2}^{-1}(F_{p_1})\subset  \dots\subset  f_{\rankp;2}^{-1}(F_{p_\ell}),\] to obtain a basis containing all non-zero vectors of the form $f_{1}(g_1^{-1}b_{1,*})$ which is compatible with the flag \eqref{eq:big-flag} in the case $k=2$.  Let $g_2$ be a matrix that maps this basis to $b_{2,*}$.

Now, apply the same process in $\C^{v_3}$, to construct a basis compatible with the flag \eqref{eq:big-flag} in the case $k=3$, and so on to construct such a basis in $\C^{v_{k+1}}$ compatible with $f_k(g_k^{-1}b_{k,*})$ and the flag \eqref{eq:big-flag} in the general case, with $g_{k+1}$ mapping this to the standard basis.

Note that in $\C^{v_{\rankp}}$, this just gives a basis compatible with $0\subset F_{p_1}\subset \cdots \subset F_{p_\ell}= \C^{v_{\rankp}}$; whereas for previous bases, we have been able to choose any bijection of our basis to the standard basis, we have to ensure that $g_{\rankp}\in P_\chi$, which is, of course, possible since our basis is compatible with the partial flag.  

Thus, the bases $\{g_k^{-1}b_{k,*}\}$ satisfy the property that $f$ sends each of these basis vectors to another basis vector or to 0.  This shows that $(g_1^{-1},\dots, g_{\rankp}^{-1})\cdot f=f_{\Sigma}$ for a segmentation $\Sigma.$ This completes the proof.
\end{proof}
\begin{lemma}\label{lem:simply-connected}
 Each of these orbits is $\Gc{\chi}$-equivariantly simply connected. 
\end{lemma}
\begin{proof}
To show this, we need only show that the stabilizer of $f_{\Sigma}$ in $\Gc{\chi}$ is simply connected.  

We can think about this stabilizer as follows.  The element $g_1$ acting on $\C^{v_1}$ must preserve the flag \eqref{eq:big-flag}, but otherwise can be chosen freely. This is a connected parabolic subgroup.  This choice determines how $g_2$ acts on the image of $f$ in $\C^{v_2}$, but on the complement $U$ to this image (given by the span of the appropriate vectors) we can freely choose an isomorphism of $U$ to itself compatible with the intersection with the flag \eqref{eq:big-flag} (again a connected parabolic), and any map $U \to \operatorname{image}(f)$ compatible with this flag (a contractible set).    

Applying this inductively, we find that the stabilizer is topologically a product of these parabolics (always connected) with affine spaces.  Paying attention to the group structure, we have a product of general linear groups (with ranks given by the multiplicities with which given flavored segments appear) extended by a unipotent radical.  Thus, as a quotient, the orbit is homotopic to a product of classifying spaces $BGL_{*}$ and, in particular, is simply connected.
\end{proof}
 \bigskip
\IndexOfNotation
\bibliography{gen}
\bibliographystyle{amsalpha}
\end{document}